\newcommand{\Rr}{\enm{\cal{R}}}

\newcommand{\Tt}{\enm{\cal{T}}}

\newcommand{\Vv}{\enm{\cal{V}}}

\renewcommand{\phi}{\varphi}
\renewcommand{\theta}{\vartheta}
\documentclass[12pt,oneside,reqno]{amsart}
 
\usepackage[all]{xy}                        %
\usepackage[hyperindex=true, colorlinks=true, linkcolor=blue, filecolor=black,
citecolor=blue, urlcolor=blue, bookmarks=true, pagebackref=true,
pdftitle={Deformations of the logarithmic tangent bundle of smooth projective surfaces}]{hyperref}

\usepackage[margin=1in]{geometry}
\usepackage[table]{xcolor}
\usepackage{colortbl, array, booktabs}

\definecolor{lightred}{rgb}{1, 0.85, 0.85}
\definecolor{lightblue}{rgb}{0.85, 0.9, 1}
\definecolor{darkline}{rgb}{0.4, 0.4, 0.4} 
\usepackage{verbatim}
\usepackage[numbers]{natbib}
\CompileMatrices                            

\UseTips                                    

\usepackage[bookmarks=true]{hyperref}       
\hypersetup{
    colorlinks=true,    
    linkcolor=blue,     
    citecolor=blue,     
    filecolor=blue,     
    urlcolor=blue       
}
\usepackage{amssymb,latexsym,amsmath,amscd}
\numberwithin{equation}{section}
\usepackage{xspace}
\usepackage{color}
\usepackage{graphicx}
\usepackage{caption}
\usepackage{mathtools}
\usepackage{enumitem}
\usepackage{tikz}
\usetikzlibrary{cd}
\usepackage[utf8]{inputenc}
\usepackage{fourier}
\usepackage{array}
\usepackage{makecell}
\usepackage{colortbl}
\usepackage{bm}
\usepackage{array}
\usepackage{subcaption}
\usepackage{longtable}
\usepackage[table]{xcolor}
\usepackage{multirow}
\usetikzlibrary{arrows.meta,positioning,calc}
\usepackage[most]{tcolorbox}
\usepackage[colorlinks=true,linkcolor=blue,citecolor=blue,urlcolor=blue]{hyperref}

\definecolor{lightred}{RGB}{255,220,220}
\definecolor{lightblue}{RGB}{220,235,255}

\newcolumntype{L}[1]{>{\raggedright\let\newline\\\arraybackslash\hspace{0pt}}m{#1}}
\newcolumntype{C}[1]{>{\centering\let\newline\\\arraybackslash\hspace{0pt}}m{#1}}
\newcolumntype{R}[1]{>{\raggedleft\let\newline\\\arraybackslash\hspace{0pt}}m{#1}}
\definecolor{lightred}{rgb}{1.0,0.9,0.9}
\definecolor{lightblue}{rgb}{0.9,0.9,1.0}

\reversemarginpar

\theoremstyle{plain}
\newtheorem{theorem}{Theorem}[section]
\newtheorem*{theorem*}{Theorem}
\newtheorem{proposition}[theorem]{Proposition}
\newtheorem{corollary}[theorem]{Corollary}
\newtheorem{lemma}[theorem]{Lemma}

\newtheorem{problem}[theorem]{Problem}
\newtheorem{claim}{Claim}[theorem]

\newenvironment{claimproof}[1][Proof of claim]{\begin{proof}[#1]}{\end{proof}}

\theoremstyle{definition}
\newtheorem{definition}[theorem]{Definition}

\newtheorem{remark}[theorem]{Remark}

\definecolor{lightred}{RGB}{255,225,225}
\definecolor{lightblue}{RGB}{222,238,255}
\definecolor{procedureblue}{RGB}{42,88,140}
\newtcolorbox{wildnessprocedure}[1][]{
  enhanced,
  breakable,
  colback=procedureblue!3,
  colframe=procedureblue,
  colbacktitle=procedureblue,
  coltitle=white,
  fonttitle=\bfseries,
  title={Finite procedure for proving reflexive toric wildness},
  boxrule=0.8pt,
  arc=2mm,
  left=3mm,
  right=3mm,
  top=2mm,
  bottom=2mm,
  before skip=10pt,
  after skip=10pt,
  attach boxed title to top left={xshift=4mm,yshift*=-\tcboxedtitleheight/2},
  boxed title style={arc=1.5mm,boxrule=0pt},
  #1
}

\newcommand{\enm}[1]{\ensuremath{#1}}          %
\newcommand{\op}[1]{\operatorname{#1}}
\newcommand{\cal}[1]{\mathcal{#1}}

\renewcommand{\bar}[1]{\overline{#1}}

\newcommand{\BB}{\enm{\mathbb{B}}}
\newcommand{\CC}{\enm{\mathbb{C}}}

\newcommand{\FF}{\enm{\mathbb{F}}}

\renewcommand{\AA}{\enm{\mathbb{A}}}
\newcommand{\PP}{\enm{\mathbb{P}}}

\newcommand{\Aa}{\enm{\cal{A}}}

\newcommand{\Ee}{\enm{\cal{E}}}
\newcommand{\Ff}{\enm{\cal{F}}}

\newcommand{\Ii}{\enm{\cal{I}}}

\newcommand{\Ll}{\enm{\cal{L}}}
\newcommand{\Mm}{\enm{\cal{M}}}
\newcommand{\Nn}{\enm{\cal{N}}}
\newcommand{\Oo}{\enm{\cal{O}}}
\renewcommand{\epsilon}{\varepsilon}

\newcommand{\Pic}{\op{Pic}}

\newcommand{\Hom}{\op{Hom}}

\newcommand{\Ext}{\op{Ext}}

\newcommand{\End}{\op{End}}
\newcommand{\Aut}{\op{Aut}}

\renewcommand{\to}[1][]{\xrightarrow{\ #1\ }}

\newcommand{\old}[1]{}

\makeatletter
\NewCommandCopy\@@pmod\pmod
\DeclareRobustCommand{\pmod}{\@ifstar\@pmods\@@pmod}
\def\@pmods#1{\mkern4mu({\operator@font mod}\mkern 6mu#1)}
\makeatother

\begin{document}
\renewcommand{\arraystretch}{1.4}
\arrayrulecolor{darkline}

\title[Deformations of the logarithmic tangent bundle of smooth projective surfaces]{Deformations of the logarithmic tangent bundle \\of smooth projective surfaces}

\author{Sukmoon Huh and Min-Gyo Jeong }

\address{Sungkyunkwan University, Suwon 440-746, Korea}
\email{sukmoonh@skku.edu}

\address{Sungkyunkwan University, Suwon 440-746, Korea}
\email{nida08@skku.edu}

\thanks{SH is supported by the National Research Foundation of Korea(NRF) grant funded by the Korea government(MSIT) (No. RS-2023-00208874). MJ is supported by Basic Science Research Program through the National Research Foundation of Korea(NRF) funded by the Ministry of Education (No. RS-2019-NR040081)}




\begin{abstract}
We study fixed-determinant deformations of the logarithmic tangent bundle associated with a smooth irreducible divisor on a smooth complex projective surface. We derive numerical criteria for the existence of positive-dimensional deformation families. Using the Enriques--Kodaira classification and Miyaoka-type inequalities, we classify the minimal surface pairs whose logarithmic tangent bundle is infinitesimally rigid. We also prove monotonicity under compatible blow-ups and compute local comparison formulas for resolutions of low-multiplicity ordinary singularities of the boundary.
\end{abstract}

\maketitle
\tableofcontents

\section{Introduction}

Let $X$ be a smooth complex projective surface and let $D\subset X$ be a nonempty divisor with  normal crossings. The logarithmic tangent bundle
\[
\Tt_X(-\log D)=\Omega_X^1(\log D)^\vee,
\]
defined as the dual of the logarithmic cotangent bundle, consists of vector fields tangent to $D$ and records the infinitesimal geometry of the pair $(X,D)$. The study of the logarithmic cotangent bundle was initiated by Deligne \cite{De} with the aim of defining a mixed Hodge structure on the complement $X\backslash D$. Later on, it was generalized by Saito to arbitrary reduced divisors in \cite{Saito}. Although it is canonically determined by the pair, it may admit deformations as a vector bundle on the fixed surface $X$. Subsequent work treated logarithmic tangent bundles as geometric vector bundles, particularly in connection with stability and Torelli problems; see \cite{DK,Valles, Angelini, BHM, UY, HMPV}.

The purpose of this paper is to determine when such deformations exist and, in particular, to classify the pairs for which the logarithmic tangent bundle is infinitesimally rigid with fixed determinant, especially when $D$ is irreducible and smooth. Deformation theory of the ordinary tangent bundle $\Tt_X$ arose in work motivated partly by vector bundles on Calabi–Yau manifolds; an early example is \cite{Huybrechts}. Peternell later formulated the systematic problem of determining when $\Tt_X$ has a genuine deformation, meaning one not obtained merely by tensoring with a varying line bundle; see \cite{Peternell}.

First-order deformations of $\Tt_X(-\log D)$ are parametrized by the first cohomology of its endomorphism bundle. The trace decomposition separates the scalar directions, which arise from tensoring by deformations of the trivial line bundle, from the trace-free directions. We therefore call the logarithmic tangent bundle \emph{genuinely infinitesimally rigid} when
\[
\mathrm{H}^1\bigl(X,\mathcal{E}\!nd_0(\Tt_X(-\log D))\bigr)=0.
\]
This is also the vanishing of the tangent space to its fixed-determinant deformation functor. We distinguish throughout between the existence of a nonzero first-order class and the stronger assertion that such a class belongs to a positive-dimensional deformation family. In this article, a positive-dimensional deformation means a positive-dimensional component through $[\Ee]$ of a fixed-determinant Kuranishi space of $\Ee$. When we assert that a family is nonisotrivial, we mean additionally that its nearby members are not all mutually isomorphic. The latter is obtained from the Kuranishi space whenever the dimension of the tangent space exceeds that of the obstruction space.

Our main result gives a complete classification on minimal surfaces.

\begin{theorem}\label{minimal-surface-pairs}
Let $X$ be a minimal smooth projective surface and let $D\subset X$ be a nonempty smooth irreducible divisor. Then $\Tt_X(-\log D)$ is genuinely infinitesimally rigid if and only if $(X,D)$ is isomorphic to one of the following.
\begin{itemize}
\item $\bigl( \PP^2, L \bigr)$ with a line $L$;
\item $\bigl(\PP^2, C\bigr)$ with a smooth conic $C$;
\item $\bigl(\FF_e, C_0\bigr)$ with $e\geq 2$,
\end{itemize}
where $C_0$ is the section of self-intersection $-e$. In every other case the logarithmic tangent bundle has a positive-dimensional fixed-determinant deformation.
\end{theorem}

This result is a logarithmic counterpart of the rigidity problem for the ordinary tangent bundle studied by Peternell~\cite{Peternell}. The presence of a nonempty boundary changes the exceptional set substantially. In particular, no minimal irrational ruled surface and no minimal surface of nonnegative Kodaira dimension carries a genuinely rigid logarithmic tangent bundle with smooth irreducible boundary. Thus all rigid pairs occur on rational surfaces and are precisely the three types displayed above.

The numerical starting point is a Riemann--Roch formula for the trace-free endomorphism bundle:
\[
\chi\bigl(X,\mathcal{E}\!nd_0(\Tt_X(-\log D))\bigr) =3\chi(\Oo_X)+(K_X+D)^2 -4\bigl(c_2(X)+K_X\cdot D+D^2\bigr).
\]
Equivalently, it can be written in terms of $K_X^2-3c_2(X)$, $K_X\cdot D$, and $D^2$. Since $\mathcal{E}\!nd_0(\Tt_X(-\log D))$ is self-dual, this formula combines with Serre duality and the fixed-determinant Kuranishi map to give effective criteria for positive-dimensional deformations. In many of the cases considered below, strict negativity of this Euler characteristic in Corollary~\ref{numerical} already proves the required nonrigidity. The limiting cases, where the Euler characteristic vanishes or may be nonnegative, require additional geometric arguments.

We prove Theorem~\ref{minimal-surface-pairs} by following the Enriques--Kodaira classification. For rational surfaces we first treat smooth hypersurfaces in projective space and then give a complete calculation on Hirzebruch surfaces. On an irrational ruled surface, the relative logarithmic tangent sequence reduces the problem to the geometry of fibers, sections, and ramified multisections. The sole numerical equality case is an \'etale multisection over an elliptic curve; after a finite \'etale base change, a transverse foliation splits the relevant sequence and supplies nontrivial trace-free endomorphisms. For surfaces of Kodaira dimension zero, Riemann--Roch settles the K3 and Enriques cases, while the abelian and bielliptic equality cases are handled by their invariant tangent directions. For Kodaira dimension one, we use the elliptic fibration, the canonical bundle formula, and the orbibundle Miyaoka--Yau--Sakai inequality in \cite[Theorem~1.3(i)]{MiyaokaOrbibundle} to treat both nef boundaries and negative multisections. Finally, on a surface of general type, logarithmic Bogomolov--Miyaoka--Yau deals with positive-genus boundaries, whereas Miyaoka's inequality for rational curves treats the remaining case.

The second theme of the paper is the behavior under birational modification in Section \ref{blow-up}. We compute the change of the logarithmic Chern numbers and of the Euler characteristic under a point blow-up. More strongly, an elementary-transform description of the logarithmic cotangent bundle shows that genuine nonrigidity is monotone under blow-ups for which the strict transform of the boundary remains smooth. If the bundle downstairs is rigid and its second trace-free endomorphism cohomology vanishes, the first blow-up already creates a genuine deformation. This gives a rigidity classification for rational pairs admitting a smooth-boundary minimalization.

\begin{theorem}\label{main-blowup}
Let $(X,D)$ be a smooth projective surface pair.
\begin{enumerate}[label=\textup{(\arabic*)}]
\item Let $(Y,D_Y)=(X_n,D_n)\to\cdots\to(X_0,D_0)=(X,D)$ be a nonempty sequence of point blow-ups, where at every stage $D_i$ is the reduced strict transform of $D_{i-1}$ and remains smooth. Then
\[
\mathrm{h}^1\bigl(Y,\mathcal E\!nd_0(\Tt_Y(-\log D_Y))\bigr)\geq\mathrm{h}^1\bigl(X,\mathcal E\!nd_0(\Tt_X(-\log D))\bigr).
\]
Consequently, infinitesimal nonrigidity of the logarithmic tangent bundle is preserved under such blow-ups. If, moreover,
\[
\mathrm{H}^2\bigl(X,\mathcal E\!nd_0(\Tt_X(-\log D))\bigr)=0,
\]
then the inequality is strict. In particular, every nontrivial such blow-up has an infinitesimally nonrigid logarithmic tangent bundle.
\item Let $\sigma\colon(Y,D)\rightarrow(X,\Delta)$ be the contraction of a $(-1)$-curve $C\subset Y$, where $D$ is
smooth and meets $C$ transversely, and put $\Delta=\sigma(D)$.

\begin{itemize}
\item [\rm (i)] If $\Delta$ has an ordinary triple point at $p=\sigma(C)$, then
\[
\mathrm{H}^1\bigl(Y,\mathcal E\!nd_0(\Omega_Y^1(\log D))\bigr)\cong\mathrm{H}^1\bigl(X,\mathcal E\!nd_0(\Omega_X^1(\log\Delta)) \bigr).
\]
\item [\rm (ii)] If $\Delta$ has an ordinary node or an ordinary quadruple point at $p$, then
\[
\begin{aligned}
&\mathrm{h}^1\bigl( Y,\mathcal E\!nd_0(\Omega_Y^1(\log D)) \bigr)\\ 
&\qquad = \mathrm{h}^1\bigl(X,\mathcal E\!nd_0(\Omega_X^1(\log\Delta))\bigr)+1-\rho_p,
\end{aligned}
\]
where $\rho_p\in\{0,1\}$ is the rank of the evaluation map associated with the corresponding local elementary transformation.
\end{itemize}
\end{enumerate}
\end{theorem}

Contractions producing singular boundaries exhibit a subtler local behavior. Using generalized logarithmic sheaves introduced in \cite{HMPV}, we obtain exact comparison formulas for ordinary double, triple, and quadruple points. Resolving an ordinary node or quadruple point changes the deformation space by a one-dimensional correction governed by an evaluation map, whereas resolving an ordinary triple point preserves the trace-free endomorphism cohomology exactly. For a tangential one-branch contraction, the exceptional restriction has nonzero first cohomology, but the required formal-neighborhood calculation remains open. The same is true for higher-multiplicity, nonordinary, and interacting exceptional configurations. Accordingly, the results of the final section are local ingredients toward, rather than a complete classification of, nonminimal surface pairs.

The paper is organized as follows. Section~2 develops the deformation theory, the logarithmic Riemann--Roch formula, and a residue criterion. Section~3 treats surfaces of Kodaira dimension minus infinity, including projective space, Hirzebruch surfaces, and irrational ruled surfaces. Sections~4--6 deal respectively with Kodaira dimensions zero, one, and two and complete the proof of the main theorem. Section~7 studies blow-ups and the local comparison problem for contractions producing singular boundary curves.

The authors would like to thank Jun-Muk Hwang for suggesting the problem studied in this paper.


\section{Preliminaries}
A \emph{smooth surface pair} $(X,D)$ is a smooth projective surface $X$ together with a smooth, possibly disconnected, reduced divisor $D$. Throughout, the base field is $\CC$. Set
\[
\Tt_X(-\log D)=\Omega_X^1(\log D)^\vee ,
\]
the logarithmic tangent bundle of rank two. Then we have $\det \Tt_X(-\log D)\cong \Oo_X(-K_X-D)$. Then the first-order deformations of $\Tt_X(-\log D)$ on the fixed surface $X$ are parametrized by
\[
\Ext^1_X(\Tt_X(-\log D),\Tt_X(-\log D))\cong \mathrm{H}^1(X,\mathcal{E}\!nd(\Tt_X(-\log D))).
\]
The trace splitting in characteristic zero gives
\[
\mathcal{E}\!nd(\Tt_X(-\log D))\cong \Oo_X\oplus\mathcal{E}\!nd_0(\Tt_X(-\log D)).
\]
The scalar summand describes first-order deformations obtained by tensoring $\Tt_X(-\log D)$ with a deformation of the trivial line bundle.

\begin{definition}
The smooth surface pair $(X,D)$ has a \emph{genuine first-order deformation} if the corresponding logarithmic tangent bundle has a genuine first-order deformation, i.e. 
\[
\mathrm{H}^1\bigl(X,\mathcal{E}\!nd_0(\Tt_X(-\log D))\bigr)\neq 0.
\]
Equivalently,
\[
\mathrm{h}^1(X,\mathcal{E}\!nd(\Tt_X(-\log D)))>q(X),
\]
where $q(X)$ is the irregularity of $X$. Note that $\mathrm{H}^1(X,\mathcal{E}\!nd_0(\Tt_X(-\log D)))$ is also the tangent space to fixed-determinant deformations of $\Tt_X(-\log D)$.
\end{definition}

The standard deformation theory of vector bundles gives a Kuranishi map with values in $\mathrm{H}^2(X,\mathcal{E}\!nd_0(\Tt_X(-\log D)))$. The following proposition is the standard Kuranishi dimension estimate
applied to the fixed-determinant deformation functor of a vector bundle; see \cite[Sections~2.1 and~2.4]{Sernesi} for the general deformation-theoretic argument and \cite[Section~4.5]{HuybrechtsLehn} for the fixed-determinant theory of sheaves.

\begin{proposition}\label{kuranishi}
If
\[
\mathrm{h}^1(X,\mathcal{E}\!nd_0(\Tt_X(-\log D)))-\mathrm{h}^2(X,\mathcal{E}\!nd_0(\Tt_X(-\log D)))>0,
\]
then the germ of the fixed-determinant deformation space of $\Tt_X(-\log D)$ contains a positive-dimensional component through $\left[\Tt_X(-\log D)\right]$. In particular, the fixed-determinant Kuranishi space of $\Tt_X(-\log D)$ has positive dimension at the distinguished point.
\end{proposition}

\begin{proof}
Let
\[
V=\mathrm{H}^1(X,\mathcal{E}\!nd_0(\Tt_X(-\log D))),\qquad W=\mathrm{H}^2(X,\mathcal{E}\!nd_0(\Tt_X(-\log D))).
\]
Fixed-determinant deformation theory associates to $\Tt_X(-\log D)$ a convergent Kuranishi map
\[
\kappa\colon (V,0)\to (W,0)
\]
whose zero germ prorepresents the fixed-determinant deformation functor up to the usual action of infinitesimal automorphisms. The constant term is zero because the origin represents $\Tt_X(-\log D)$, and the differential at the origin is zero because $V$ is the full Zariski tangent space; the first possibly nonzero term is the quadratic obstruction map given by the Yoneda bracket.

Choose linear coordinates on $W$. Then $\kappa^{-1}(0)$ is locally cut out inside the smooth germ $(V,0)$ by at most $\dim W$ holomorphic equations. Krull's principal ideal theorem, applied successively to these equations in the regular local ring $\Oo_{V,0}$, gives
\[
\dim_0\kappa^{-1}(0)\geq \dim V-\dim W=\mathrm{h}^1(X,\mathcal{E}\!nd_0(\Tt_X(-\log D)))-\mathrm{h}^2(X,\mathcal{E}\!nd_0(\Tt_X(-\log D))).
\]
If this number is positive, an irreducible component of the zero germ through the origin has positive dimension. Any nonzero tangent direction on that component is trace-free and therefore gives a genuine fixed-determinant first-order deformation of $\Tt_X(-\log D)$.
\end{proof}

\begin{remark}
The conclusion of Proposition \ref{kuranishi} does not assert that the entire Kuranishi germ is smooth. For that reason we use the phrase ``positive-dimensional deformation'' rather than calling every such deformation unobstructed.
\end{remark}

\begin{remark}
Three deformation problems must not be conflated.
\begin{enumerate}[label=\textup{(\arabic*)}]
\item Deformations of $\Tt_X(-\log D)$ on the fixed $X$ are controlled by $\mathrm{H}^i(X,\mathcal{E}\!nd(\Tt_X(-\log D)))$.
\item Locally trivial deformations of the pair $(X,D)$ are controlled by $\mathrm{H}^i(X,\Tt_X(-\log D))$.
\item The logarithmic Atiyah class differentiates the assignment $(X,D)\mapsto \Tt_X(-\log D)$ and produces a natural map from infinitesimal pair deformations to bundle deformations.
\end{enumerate}
The third problem asks which elements of $\mathrm{H}^1(X,\mathcal{E}\!nd_0(\Tt_X(-\log D)))$ are geometric, in the sense of arising from a moving pair: this is closely related to the infinitesimal logarithmic Torelli map. A complete treatment should compute its kernel and cokernel rather than identify the two deformation spaces.
\end{remark}

The logarithmic tangent sequence is
\begin{equation}\label{log-tangent}
0\to \Tt_X(-\log D)\to \Tt_X \to \Oo_D(D)\to 0.
\end{equation}
In the Grothendieck group, $[\Oo_D(D)]=[\Oo_X(D)]-[\Oo_X]$. It follows that

\begin{equation}\label{chern}
c_1(\Tt_X(-\log D))=-K_X-D, \qquad
c_2(\Tt_X(-\log D))=c_2(X)+K_X\cdot D+D^2.
\end{equation}
We write
\[
\bar c_1^2(X,D)=(K_X+D)^2, \qquad \bar c_2(X,D)=c_2(X)+K_X\cdot D+D^2.
\]
For a general bundle $\Ff$ of rank two, one has
\[
c_1(\mathcal{E}\!nd_0(\Ff))=0, \qquad c_2(\mathcal{E}\!nd_0(\Ff))=4c_2(\Ff)-c_1(\Ff)^2.
\]

\begin{theorem}\label{RR}
For every smooth projective surface pair $(X,D)$, one obtains
\begin{equation}\label{RR-log}
\chi\bigl(X,\mathcal{E}\!nd_0(\Tt_X(-\log D))\bigr) =3\chi(\Oo_X)+\bar c_1^2(X,D)-4\bar c_2(X,D).
\end{equation}
Equivalently,
\begin{equation}
\chi\bigl(X,\mathcal{E}\!nd_0(\Tt_X(-\log D))\bigr) =\frac54\bigl(K_X^2-3c_2(X)\bigr) -2K_X\cdot D-3D^2.
\end{equation}
\end{theorem}

\begin{proof}
For a bundle $\Ff$ of rank $r$ on $X$, Hirzebruch--Riemann--Roch gives
\[
\chi(\Ff)=r\chi(\Oo_X)+\frac12c_1(\Ff)\cdot(c_1(\Ff)-K_X)-c_2(\Ff).
\]
Taking $\Ff=\mathcal{E}\!nd_0(\Tt_X(-\log D))$, whose rank is three and whose first Chern class is
zero, yields
\begin{align*}
\chi(\mathcal{E}\!nd_0(\Tt_X(-\log D)))&=3\chi(\Oo_X)-c_2(\mathcal{E}\!nd_0(\Tt_X(-\log D)))\\
&=3\chi(\Oo_X)+\bar c_1^2(X,D)-4\bar c_2(X,D).
\end{align*}
Expanding the last line gives
\[
3\chi(\Oo_X)+K_X^2-4c_2(X)-2K_X\cdot D-3D^2.
\]
Finally Noether's identity $12\chi(\Oo_X)=K_X^2+c_2(X)$ transforms its first three terms into
\[
\frac14(K_X^2+c_2(X))+K_X^2-4c_2(X)=\frac54(K_X^2-3c_2(X)). \qedhere
\]
\end{proof}

The bundle $\mathcal{E}\!nd_0(\Tt_X(-\log D))$ is self-dual under the trace pairing. Therefore, Serre duality gives
\[
\mathrm{H}^2(X,\mathcal{E}\!nd_0(\Tt_X(-\log D)))^\vee \cong \mathrm{H}^0(X,\mathcal{E}\!nd_0(\Tt_X(-\log D))\otimes K_X).
\]

\begin{corollary}\label{numerical}
Each of the following conditions implies a genuine positive-dimensional deformation:
\begin{enumerate}[label=\textup{(\arabic*)}]
\item $\chi(X,\mathcal{E}\!nd_0(\Tt_X(-\log D)))<0$;
\item $4\bar c_2(X,D)>(K_X+D)^2+3\chi(\Oo_X)$;
\item more generally, $h^0(X,\mathcal{E}\!nd_0(\Tt_X(-\log D)))>\chi(X,\mathcal{E}\!nd_0(\Tt_X(-\log D)))$.
\end{enumerate}
\end{corollary}

\begin{proof}
Setting $\Ff=\mathcal{E}\!nd_0(\Tt_X(-\log D))$, Theorem \ref{RR} yields
\[
\mathrm{h}^1(\Ff)-\mathrm{h}^2(\Ff)=h^0(\Ff)-3\chi(\Oo_X)-\bar c_1^2(X,D)+4\bar c_2(X,D).
\]
If $\chi(\Ff)<0$, then $\mathrm{h}^0(\Ff)-\chi(\Ff)>0$ because $\mathrm{h}^0(\Ff)\geq0$. Condition~\textup{(2)} is exactly the strict negativity of the right-hand
side of \eqref{RR-log}; hence it implies condition~\textup{(1)}. Condition~\textup{(3)} says directly that $\mathrm{h}^0(\Ff)-\chi(\Ff)>0$. Thus in every case $\mathrm{h}^1(\Ff)-\mathrm{h}^2(\Ff)>0$, and Proposition~\ref{kuranishi} produces a positive-dimensional fixed-determinant deformation.
\end{proof}

\begin{remark}
The logarithmic Bogomolov--Miyaoka--Yau inequality $(K_X+D)^2\leq 3\bar c_2(X,D)$, when applicable, does not by itself imply condition~\textup{(2)} of Corollary~\ref{numerical}. The additional term $3\chi(\Oo_X)$ is essential.
\end{remark}

Let $A$ be a smooth divisor meeting $D$ transversely. The residue sequence for the enlarged simple-normal-crossing divisor $D+A$ is
\[
0\to \Omega_X^1(\log D) \to \Omega_X^1(\log(D+A)) \to \Oo_A\to 0.
\]
Tensoring by $\Tt_X(-\log D)$ gives
\begin{equation}\label{residue-E}
0\to \mathcal{E}\!nd(\Tt_X(-\log D)) \to  \Tt_X(-\log D)\otimes\Omega_X^1(\log(D+A)) \to \Tt_X(-\log D)_{|A}\to 0.
\end{equation}

\begin{proposition}\label{residue}
Assume that
\[
\mathrm{H}^0(X,\mathcal{E}\!nd(\Tt_X(-\log D)))\cong \mathrm{H}^0\bigl(X,\Tt_X(-\log D)\otimes\Omega_X^1(\log(D+A))\bigr).
\]
If $\mathrm{H}^0(A,\Tt_X(-\log D)_{|A})\neq0$, then $\mathrm{H}^1(X,\mathcal{E}\!nd(\Tt_X(-\log D)))\neq0$.
\end{proposition}

\begin{proof}
The beginning of the long exact cohomology sequence of
\eqref{residue-E} is
\begin{align*}
0&\to \mathrm{H}^0(X,\mathcal{E}\!nd(\Tt_X(-\log D))) \to \mathrm{H}^0\bigl(X,\Tt_X(-\log D)\otimes\Omega_X^1(\log(D+A))\bigr)\\
&\to \mathrm{H}^0(A,\Tt_X(-\log D)_{|A}) \xrightarrow{~~\partial~~}\mathrm{H}^1(X,\mathcal{E}\!nd(\Tt_X(-\log D))).
\end{align*}
The assumption shows that $\partial$ is injective, and so the assertion follows.
\end{proof}

For the genuine version, the identity section determines
\[
\Oo_X\subset \mathcal{E}\!nd(\Tt_X(-\log D)) \subset \Tt_X(-\log D)\otimes\Omega_X^1(\log(D+A)).
\]
Setting $\Rr_A:=\Tt_X(-\log D)\otimes\Omega_X^1(\log(D+A)) \big/ \Oo_X$, we have 
\begin{equation}\label{residue-zero-seq}
0\to \mathcal{E}\!nd_0(\Tt_X(-\log D)) \to \Rr_A\to \Tt_X(-\log D)_{|A}\to 0.
\end{equation}

\begin{proposition}\label{residue-zero}
If 
\[
\mathrm{H}^0(X,\mathcal{E}\!nd_0(\Tt_X(-\log D)))\cong \mathrm{H}^0(X,\Rr_A)
\]
and $\mathrm{H}^0(A,\Tt_X(-\log D)_{|A})\neq0$, then
$\mathrm{H}^1(X,\mathcal{E}\!nd_0(\Tt_X(-\log D)))\neq0$.
\end{proposition}

\begin{proof}
The same argument in the proof of Proposition \ref{residue} works verbatim with \eqref{residue-zero-seq}. 
\end{proof}


\section{Surfaces of \texorpdfstring{$\kappa=-\infty$}{kappa = -infty}}
\subsection{Projective space}
Let $D_d\subset\mathbb P^n$ be a smooth hypersurface of degree $d$, and write $H=c_1(\Oo_{\PP^n}(1))$.  We use the logarithmic tangent convention
\[
\bar c_i(\PP^n,D_d) :=c_i\bigl(\Tt_{\PP^n}(-\log D_d)\bigr).
\]
The logarithmic tangent sequence and the Euler sequence give
\[
c\bigl(\Tt_{\PP^n}(-\log D_d)\bigr) =\frac{(1+H)^{n+1}}{1+dH},
\]
and it follows that
\begin{align}
\bar c_1(\PP^n,D_d) &=-K_{\PP^n}-D_d=(n+1-d)H, \\
\bar c_2(\mathbb P^n,D_d)&=\left[\binom{n+1}{2}-(n+1)d+d^2\right]H^2;
\end{align}
see \cite{DK}. For $n=2$, putting $\Ee_{2,d}=\Tt_{\PP^2}(-\log D_d)$, these formulas specialize to
\[
\bar c_1^2(\PP^2,D_d)=(d-3)^2, \qquad \bar c_2(\PP^2,D_d)=d^2-3d+3.
\]
Thus Theorem~\ref{RR} becomes
\begin{equation}\label{plane-chi}
\chi\bigl(\PP^2,\mathcal{E}\!nd_0(\Ee_{2,d})\bigr)=-3d(d-2).
\end{equation}
We now treat all projective dimensions simultaneously; the calculation above will give the sharper Kuranishi estimate in dimension two.

\begin{theorem}\label{proj-hypersf}
For a smooth hypersurface $D_d\subset\mathbb P^n$ with $n\ge 2$ of
degree $d\geq1$, put
\[
\Ee_{n,d}=\Tt_{\PP^n}(-\log D_d).
\]
Then $\mathrm{H}^1\bigl(\PP^n,\mathcal{E}\!nd_0(\Ee_{n,d})\bigr)=0$ if and only if $d\in\{1,2\}$. More precisely:
\begin{enumerate}[label=\textup{(\arabic*)}]
\item if $d=1$, then $\Ee_{n,1}\cong \Oo_{\PP^n}(1)^{\oplus n}$; \smallskip
\item if $d=2$, then $\Ee_{n,2}\cong\Omega_{\PP^n}^1(2)$; \smallskip
\item if $d\geq3$ and $n\geq3$, then
\begin{equation}\label{ext-lower-bound}
\mathrm{h}^1\bigl(\PP^n,\mathcal{E}\!nd_0(\Ee_{n,d})\bigr)\geq (n+1) \left[\binom{n+d-1}{n}-(n+1)\right]>0;
\end{equation}
 \item if $d\geq 3$ and $n=2$, then
\begin{equation}\label{plane-kuranishi-bound}
\mathrm{h}^1(\mathcal{E}\!nd_0(\Ee_{2,d}))-\mathrm{h}^2(\mathcal{E}\!nd_0(\Ee_{2,d}))=\mathrm{h}^0(\mathcal{E}\!nd_0(\Ee_{2,d}))+3d(d-2),
\end{equation}
so every fixed-determinant Kuranishi space has local dimension at least $3d(d-2)$.
\end{enumerate}
\end{theorem}

\begin{proof}
Let $f\in \mathrm{H}^0(\PP^n,\Oo_{\PP^n}(d))$ define $D_d$. Since $D_d$ is smooth, its partial derivatives have no common zero in
$\PP^n$, and the logarithmic tangent bundle therefore has the gradient resolution
\begin{equation}\label{grad-res}
0\to \Ee_{n,d} \to \Oo_{\PP^n}(1)^{\oplus(n+1)} \xrightarrow{(\partial f/\partial x_0,\ldots,\partial f/\partial x_n)}
 \Oo_{\PP^n}(d)\to 0.
\end{equation}
For $d=1$, a projective change of coordinates makes the gradient a nonzero constant vector, and \eqref{grad-res} gives $\Ee_{n,1}\cong \Oo_{\PP^n}(1)^{\oplus n}$. For $d=2$, the gradient consists of independent linear forms. After a linear change of coordinates, \eqref{grad-res} is the Euler sequence twisted by $\Oo_{\PP^n}(1)$, and hence $\Ee_{n,2}\cong \Omega_{\PP^n}^1(2)$; see \cite{DK, Angelini}.

The bundle $\Ee_{n,1}$ is a direct sum of copies of one line bundle, so its trace-free endomorphism bundle is trivial and has vanishing first cohomology. The second is a twist of the exceptional bundle $\Omega_{\PP^n}^1$; equivalently, Bott's theorem applied to its endomorphism bundle gives
\[
\mathrm{H}^1\bigl(\PP^n,\mathcal{E}\!nd_0(\Omega_{\PP^n}^1(2))\bigr)=0.
\]

\noindent Suppose now that $n\geq3$ and $d\geq3$. Dualizing \eqref{grad-res} and twisting first by $\Oo_{\PP^n}(1)$
and then by $\Oo_{\PP^n}(d-1)$ gives
\begin{align*}
0&\to \Oo_{\PP^n}(1-d)\to\Oo_{\PP^n}^{\oplus(n+1)}\to \Ee_{n,d}^\vee (1)\to 0;\smallskip \\
0&\to \Oo_{\PP^n}\to \Oo_{\PP^n}(d-1)^{\oplus(n+1)}\to \Ee_{n,d}^\vee (d)\to 0.
\end{align*}
The intermediate cohomology of line bundles on $\PP^n$ vanishes so that we obtain
\begin{align*}
\mathrm{h}^0(\Ee_{n,d}^\vee(1))=n+1,\qquad \mathrm{h}^0(\Ee_{n,d}^\vee(d))=(n+1)N-1,\qquad \Ext_{\PP^n}^1(\Ee_{n,d},\Oo_{\PP^n}(1))=0,
\end{align*}
where $N=\binom{n+d-1}{n}$. Applying $\Hom_{\PP^n}(\Ee_{n,d},-)$ to \eqref{grad-res} consequently yields
\begin{align*}
\mathrm{h}^1(\mathcal{E}\!nd_0(\Ee_{n,d}))=\mathrm{h}^1(\mathcal{E}\!nd(\Ee_{n,d}))&=(n+1)N-1-(n+1)^2+\mathrm{h}^0(\mathcal{E}\!nd(\Ee_{n,d}))\\
&\geq(n+1)\bigl(N-(n+1)\bigr)>0
\end{align*}
It remains to treat $n=2$ and $d\geq3$. Formula \eqref{plane-chi} is strictly negative in this range, so Corollary~\ref{numerical} gives a genuine positive-dimensional deformation. Moreover, \eqref{plane-chi} gives the lower bound $3d(d-2)$ for the local dimension by Proposition~\ref{kuranishi}.
\end{proof}

\begin{remark}
If $n=1$, a smooth divisor of degree $d$ is a reduced collection of $d$ points and
\[
\Tt_{\PP^1}(-\log D_d)\cong \Oo_{\PP^1}(2-d).
\]
This has rank one, so its trace-free endomorphism bundle is zero for every $d$. Thus the theorem necessarily begins in dimension two.
\end{remark}

\begin{remark}
Deforming the equation of $D_d$ produces a family of logarithmic bundles, but the map from embedded deformations of the curve to the deformation space of $\Ee_{n,d}$ can have a kernel: distinct divisors may have isomorphic logarithmic bundles. This is the {\it logarithmic Torelli problem}. The numerical lower bound above concerns all bundle deformations, not only those induced by moving $D_d$.
\end{remark}

\subsection{Hirzebruch surface}
Let $\FF_e=\PP (\Oo_{\PP^1}\oplus\Oo_{\PP^1}(e))$ with $e\geq 0$, and denote by $C_0$ its section with $C_0^2=-e$ and by $f$ the class of a fiber. Thus we have $K_{\FF_e}=-2C_0-(e+2)f$. 

\begin{theorem}\label{hirzebruch-classification}
For a nonempty smooth irreducible divisor $D\subset\mathbb F_e$, we have 
\[
\mathrm{H}^1\bigl(\FF_e,\mathcal{E}\!nd_0(\Tt_{\FF_e}(-\log D))\bigr)=0
\]
if and only if $e\geq1$ and $D=C_0$ is the negative section. For $D=C_0$ one has, more precisely,
\[
\mathrm{h}^0(\FF_e,\mathcal{E}\!nd_0(\Tt_{\FF_e}(-\log C_0)))=e-1, \qquad \mathrm{h}^i(\FF_e,\mathcal{E}\!nd_0(\Tt_{\FF_e}(-\log C_0)))=0, \qquad \text{for }~i=1,2.
\]
\end{theorem}

\begin{proof}
Write $D\equiv aC_0+bf$, where $a=D\cdot f\geq0$. The logarithmic Riemann--Roch formula becomes
\begin{equation}\label{hirzebruch-chi}
\chi\bigl(\mathcal{E}\!nd_0(\Tt_{\FF_e}(-\log D))\bigr)=4a-5+(3a^2-2a)e+(4-6a)b.
\end{equation}
If $a=0$, irreducibility and reducedness imply that $D$ is a fiber, so $b=1$ and \eqref{hirzebruch-chi} equals $-1$. Suppose that $a\geq1$. If $D\neq C_0$, then the two distinct irreducible curves $D$ and $C_0$ have nonnegative intersection. Hence $b-ae=D\cdot C_0\geq0$. For $e\geq1$, since $4-6a<0$, \eqref{hirzebruch-chi} gives
\begin{align*}
\chi\bigl(\mathcal{E}\!nd_0(\Tt_{\FF_e}(-\log D))\bigr)&\leq 4a-5+(3a^2-2a)e+(4-6a)ae\\
&=4a-5-ea(3a-2)<0:
\end{align*}
for $a=1$ the last expression is $-1-e$, while for $a\geq2$ it is at most $-3(a-1)^2-2$. If $e=0$, a smooth irreducible divisor with $a=0$ or $b=0$ is a fiber of one of the two rulings, and its Euler characteristic is $-1$. If $a,b\geq1$, then
\[
\chi\bigl(\mathcal{E}\!nd_0(\Tt_{\FF_e}(-\log D))\bigr)=4a-5+(4-6a)b<0.
\]
Thus every case except possibly $e\geq1$ and $D=C_0$ has negative Euler characteristic, and Corollary~\ref{numerical} gives $\mathrm{H}^1\bigl(\mathcal{E}\!nd_0(\Tt_{\FF_e}(-\log D))\bigr)\neq0$.

It remains to calculate the exceptional case. Let $\pi\colon\FF_e\rightarrow \PP^1$ be the ruling and assume $e\geq1$. Projection to the base gives the logarithmic relative tangent sequence
\begin{equation}\label{hirzebruch-log-relative}
0\to \Oo_{\FF_e}(C_0+ef)\to \Tt_{\FF_e}(-\log C_0) \to \Oo_{\FF_e}(2f) \to 0.
\end{equation}
Indeed, $\Tt_{\FF_e/\PP^1}=\Oo_{\FF_e}(2C_0+ef)$, and relative vector fields tangent to $C_0$ vanish normally along $C_0$. The extension group is
\begin{align*}
\Ext_{\FF_e}^1(\Oo_{\FF_e}(2f), \Oo_{\FF_e}(C_0+ef))&\cong \mathrm{H}^1\bigl(\FF_e,\Oo_{\FF_e}(C_0+(e-2)f)\bigr)\\
&\cong \mathrm{H}^1\bigl(\PP^1, \Oo_{\PP^1}(e-2)\oplus\Oo_{\PP^1}(-2)\bigr) \cong \CC.
\end{align*}
The extension \eqref{hirzebruch-log-relative} is nonsplit. Otherwise its pushout by $\Tt_{\FF_e/\PP^1}(-C_0)\hookrightarrow \Tt_{\FF_e/\PP^1}$ would split the ordinary relative tangent sequence
\[
0\to \Tt_{\FF_e/\PP^1} \to \Tt_{\FF_e} \to \pi^* \Tt_{\PP^1} \to 0.
\]
This would give the nontrivial ruled surface $\FF_e$ a holomorphic connection over $\PP^1$, which is impossible for $e>0$. Hence \eqref{hirzebruch-log-relative} represents the unique nonzero extension class up to scale.

Now the restriction to a fiber shows
\[
\Hom_{\FF_e}(\Oo_{\FF_e}(C_0+ef),\Oo_{\FF_e}(2f))\cong \mathrm{H}^0(\FF_e,\Oo_{\FF_e}(-C_0+(2-e)f))=0,
\]
whereas
\begin{align*}
\Hom_{\FF_e}(\Oo_{\FF_e}(2f), \Oo_{\FF_e}(C_0+ef))\cong \mathrm{H}^0\bigl(\FF_e,\Oo_{\FF_e}(C_0+(e-2)f)\bigr) \cong \mathrm{H}^0\bigl(\PP^1, \Oo_{\PP^1}(e-2)\oplus\Oo_{\PP^1}(-2)\bigr),
\end{align*}
which has dimension $e-1$. Every endomorphism of $\Tt_{\FF_e}(-\log C_0)$ preserves $\Oo_{\FF_e}(C_0+ef)$. Its two diagonal scalars must agree because the extension class is nonzero; conversely, every element of $\Hom_{\FF_e}(\Oo_{\FF_e}(2f), \Oo_{\FF_e}(C_0+ef))$ gives a nilpotent endomorphism. Therefore
\[
\mathrm{h}^0\bigl(\mathcal{E}\!nd(\Tt_{\FF_e}(-\log C_0))\bigr)=1+(e-1)=e, \qquad \mathrm{h}^0\bigl(\mathcal{E}\!nd_0(\Tt_{\FF_e}(-\log C_0))\bigr)=e-1.
\]
Finally, on every fiber $f\cong\PP^1$,
\[
\Tt_{\FF_e}(-\log C_0)_{|f}\cong \Oo_{\PP^1}(1)\oplus\Oo_{\PP^1}, \qquad (K_{\FF_e})_{|f}\cong \Oo_{\PP^1}(-2).
\]
Consequently,
\[
 (\mathcal{E}\!nd_0(\Tt_{\FF_e}(-\log C_0))\otimes K_{\FF_e})_{|f}\cong \Oo_{\PP^1}(-1)\oplus\Oo_{\PP^1}(-2)\oplus\Oo_{\PP^1}(-3),
\]
which has no sections. Thus, by Serre duality,
\[
\mathrm{h}^2\bigl(\mathcal{E}\!nd_0(\Tt_{\FF_e}(-\log C_0))\bigr)=\mathrm{h}^0\bigl(\mathcal{E}\!nd_0(\Tt_{\FF_e}(-\log C_0))\otimes K_{\FF_e}\bigr)=0.
\]
Equation \eqref{hirzebruch-chi} gives $\chi(\mathcal{E}\!nd_0(\Tt_{\FF_e}(-\log C_0)))=e-1$. Hence
\[
\mathrm{h}^1\bigl(\mathcal{E}\!nd_0(\Tt_{\FF_e}(-\log C_0))\bigr)=(e-1)+0-(e-1)=0,
\]
as required.
\end{proof}

\begin{corollary}\label{rational-compatible-classification}
Let $(Y,D_Y)$ be a smooth rational surface pair with nonempty smooth irreducible boundary. Assume that there is a sequence of $(-1)$-curve contractions
\[
(Y,D_Y)=(Y_n,D_n)\to\cdots\to (Y_0,D_0)
\]
such that every $D_i$ is smooth and $Y_0$ is either $\PP^2$ or $\FF_e$. Then
\[
\mathrm{H}^1\bigl(Y,\mathcal{E}\!nd_0(\Tt_Y(-\log D_Y))\bigr)=0
\]
if and only if $n=0$ and one of the following holds:
\[
(Y,D_Y)\cong(\PP^2,L),\qquad (Y,D_Y)\cong(\PP^2,Q),\qquad (Y,D_Y)\cong(\FF_e,C_0)\quad \text{for }e\geq1,
\]
where $L$ is a line and $Q$ is a smooth conic.
\end{corollary}

\begin{proof}
For $n=0$, use Theorems~\ref{proj-hypersf} and \ref{hirzebruch-classification}. If the model pair is already nonrigid, monotonicity in Theorem~\ref{log-blowup-comparison} preserves a nonzero first cohomology group through the sequence. For each of the three rigid model pairs, the proof of Theorems~\ref{proj-hypersf} and \ref{hirzebruch-classification} gives $\mathrm{H}^2(\mathcal{E}\!nd_0(\Tt_{Y_0}(-\log D_0)))=0$. Hence the first blow-up is nonrigid by Corollary~\ref{compatible-blowups}, and nonrigidity again persists.
\end{proof}

\begin{remark}
The hypothesis that all $D_i$ remain smooth is substantive. Contracting a $(-1)$-curve that meets $D_i$ twice, or is tangent to it, produces a singular image boundary. Completing the classification for arbitrary nonminimal rational pairs therefore requires the generalized logarithmic sheaf for the resulting singular curve and cannot be deduced from the smooth-pair comparison alone.
\end{remark}


\subsection{Irrational ruled surface}
Let $f\colon X=\PP_C(\Vv)\rightarrow C$ be a minimal ruled surface over a curve of genus $g\geq1$. Choose a normalized section $C_0$ with $C_0^2=-e$ and let $F$ denote a fiber. Thus
\[
K_X=-2C_0+(2g-2-e)F,\qquad c_2(X)=4-4g.
\]
Write a smooth irreducible divisor as $D\equiv aC_0+bF$, where $a=D\cdot F\geq0$. The logarithmic Riemann--Roch formula becomes 
\begin{equation}\label{irr-ruled-chi}
\chi\bigl(\mathcal{E}\!nd_0(\Tt_X(-\log D))\bigr)=(5-4a)(g-1)+ea(3a-2)+(4-6a)b.
\end{equation}
Indeed, the intersection table is
\[
C_0^2=-e,\qquad C_0\cdot F=1,\qquad F^2=0.
\]
It gives
\begin{align*}
 K_X\cdot D
 &=(-2C_0+(2g-2-e)F)\cdot(aC_0+bF)\\
 &=a(e+2g-2)-2b,\\
 D^2&=-ea^2+2ab.
\end{align*}
Furthermore $K_X^2=8(1-g)$ and $c_2(X)=4(1-g)$, so
\[
\frac54(K_X^2-3c_2(X))=5(g-1).
\]
Substitution of these three identities into the formula of Theorem~\ref{RR}, followed by collecting the coefficients of $g-1$, $e$, and $b$, is exactly \eqref{irr-ruled-chi}.



We first record the relative logarithmic tangent sequence of the
ruling. Its ramification term will also be used in the multisection
case.
\begin{lemma}\label{relative-logarithmic-sequence}
Let $f\colon X\to C$ be a ruled surface, and let $D\subset X$ be a smooth horizontal irreducible curve. Let $Z\subset D$ be the ramification subscheme of the finite morphism $f_{|D}\colon D\rightarrow C$, viewed as a zero-dimensional subscheme of $X$. Then the differential of $f$ induces an exact sequence
\begin{equation}\label{irr-ruled-relative-log}
0\to\Tt_{X/C}(-D) \to\Tt_X(-\log D) \to f^*\Tt_C\otimes\Ii_{Z,X} \to0.
\end{equation}
Moreover, $\mathrm{length}(Z)=\deg R$, where $R$ is the ramification divisor of $f_{|D}$.
\end{lemma}

\begin{proof}
Let $p\in D$ be a point at which $f_{|D}$ has ramification index $r$. After choosing suitable analytic coordinates $(x,t)$ on $X$ centered at $p$, we may write
\[
f(x,t)=t,\qquad D=(t-x^r=0).
\]
A local vector field
\[
v=A(x,t)\frac{\partial}{\partial x}  +B(x,t)\frac{\partial}{\partial t}
\]
belongs to $\Tt_X(-\log D)$ if and only if
\[
v(t-x^r)=B-rx^{r-1}A\in(t-x^r).
\]
It follows that $B\in(x^{r-1},t-x^r)$. Conversely, if
\[
B=x^{r-1}C+(t-x^r)H,
\]
then, since the ground field has characteristic zero, taking $A=C/r$ gives
\[
B-rx^{r-1}A=(t-x^r)H.
\]
Thus every local section of
\[
\Ii_{Z,X}\cdot\frac{\partial}{\partial t},
\qquad
\Ii_{Z,X}=(x^{r-1},t-x^r),
\]
is obtained as the base component of a logarithmic vector field.

The kernel consists of logarithmic vector fields with $B=0$. In that case,
\[
rx^{r-1}A\in(t-x^r).
\]
Since $t-x^r$ does not divide $x$, this implies $A\in(t-x^r)$. Thus the vertical vector field vanishes along $D$, and the kernel is precisely $\Tt_{X/C}(-D)$. Finally,
\[
\mathrm{length}\left(\Oo_{X,p}/(x^{r-1},t-x^r)\right)=r-1,
\]
which is the ramification multiplicity at $p$. Summing over the ramification points proves both the exactness of \eqref{irr-ruled-relative-log} and the asserted length of $Z$.
\end{proof}
If the boundary is a fiber $D=F_p$, it is not horizontal, and the corresponding logarithmic sequence is instead
\begin{equation}\label{irr-ruled-fiber-log}
0\to \Tt_{X/C} \to \Tt_X(-\log F_p)\to f^*\Tt_C(-F_p)\to 0.
\end{equation}

\begin{remark}\label{elem-observation}
Consider an extension
\begin{equation}\label{irr-ruled-two-line-extension}
0\to \Ll \to \Ee \to \Mm \to 0
\end{equation}
of line bundles on $X$, and set $\Aa:=\Ll\otimes \Mm^{-1}$. Every homomorphism $\Mm\rightarrow \Ll$ defines the trace-free nilpotent endomorphism $\Ee \rightarrow \Mm \rightarrow \Ll \rightarrow \Ee$. Hence we obtain an injection $\Aa \hookrightarrow \mathcal{E}\!nd_0(\Ee)$. If the extension splits, the diagonal endomorphism $(1,-1)$ supplies one additional independent section. If it does not split, its nonzero extension class belongs to
\[
\Ext_X^1(\Mm,\Ll) \cong \mathrm{H}^1(X,\Aa),
\]
so $\mathrm{h}^1(X,\Aa)\geq1$.

\end{remark}

\begin{proposition}\label{irr-ruled-fibers-sections}
Let $X=\PP_C(\Vv)$ with $g(C)=g\geq1$, and let $D$ be either a fiber or a section of $f$. Then we have
\[
\mathrm{H}^1\bigl(X,\mathcal{E}\!nd_0(\Tt_X(-\log D))\bigr)\neq0.
\]
In fact, $\Tt_X(-\log D)$ has a positive-dimensional fixed-determinant deformation.
\end{proposition}

\begin{proof}
First suppose that $D=F_p$ is a fiber. \eqref{irr-ruled-fiber-log} has the form \eqref{irr-ruled-two-line-extension} with
\[
\Ll=\Tt_{X/C},\qquad \Mm=f^*\Tt_C(-F_p),\qquad \Aa=\Tt_{X/C}\otimes f^*K_C(F_p).
\]
Numerically, we have
\[
\Aa \equiv 2C_0+(e+2g-1)F, \qquad \Aa+K_X\equiv(4g-3)F,
\]
from which we obtain $\chi(X,\Aa)=3g$ by Riemann--Roch. Moreover, $\Aa_{|F}\cong \Oo_{\PP^1}(2)$, so $\mathrm{H}^2(X,\Aa)=0$ by Serre duality and restriction to a general fiber. Consequently
\[
\mathrm{h}^0(X,\Aa)=3g+\mathrm{h}^1(X,\Aa).
\]
Using $\Tt_{X/C}\cong K_X^{-1}\otimes f^*K_C$, we obtain
\[
\Aa\otimes K_X\cong f^*\bigl(K_C^{\otimes2}(p)\bigr).
\]
The line bundle $K_C^{\otimes2}(p)$ has degree $4g-3>2g-2$. Therefore, Riemann--Roch on $C$ gives
\[
\mathrm{h}^0(X,\Aa\otimes K_X)=\mathrm{h}^0\bigl(C,K_C^{\otimes2}(p)\bigr)=3g-2.
\] 
The injection $\Aa \hookrightarrow\mathcal{E}\!nd_0(\Tt_X(-\log F_p))$ and its twist by $K_X$, together with Serre duality, therefore give
\begin{align*}
\mathrm{h}^0(X,\mathcal{E}\!nd_0(\Tt_X(-\log F_p)))&\geq3g+\mathrm{h}^1(X,\Aa)+\varepsilon,\\
\mathrm{h}^2(X,\mathcal{E}\!nd_0(\Tt_X(-\log F_p)))&\geq3g-2,
\end{align*}
where $\varepsilon=1$ if the extension splits and $\varepsilon=0$ otherwise. Formula \eqref{irr-ruled-chi}, with $a=0$ and $b=1$, gives $\chi(\mathcal{E}\!nd_0(\Tt_X(-\log F_p)))=5g-1$. Hence
\[
\mathrm{h}^1(X,\mathcal{E}\!nd_0(\Tt_X(-\log F_p))) \geq g-1+\mathrm{h}^1(X,\Aa)+\varepsilon.
\]
If the extension is nonsplit, then $\mathrm{h}^1(X,\Aa)\geq1$; if it splits, then $\varepsilon=1$. The right-hand side is therefore positive in either case. This proves that $\Tt_X(-\log F_p)$ is infinitesimally nonrigid. It does not yet prove that its fixed-determinant Kuranishi space is positive-dimensional, because the preceding estimate gives no upper bound for $\mathrm{h}^2(X,\mathcal{E}\!nd_0(\Tt_X(-\log F_p)))$. We therefore construct a nonisotrivial fixed-determinant family directly. 

Consider a map $\alpha\colon C\rightarrow\Pic^0(C)$, defined by $p\mapsto\Oo_C(p-p_0)$. Since multiplication by two on $\Pic^0(C)$ is finite \'etale, there are an \'etale neighborhood $U\to C$ of $p_0$ and an algebraic family of line bundles $Q_u\in\Pic^0(C)$, normalized by $Q_{u_0}\cong \Oo_C$, such that
\[
Q_u^{\otimes2}\cong \Oo_C(p(u)-p_0).
\]
The relative logarithmic tangent bundles of the incidence divisor
\[
\Ff=\{(x,u)\in X\times U \mid f(x)=p(u)\}
\]
form an algebraic family $\Ee_u:=\Tt_X(-\log F_{p(u)})$. Set
\[
\widetilde \Ee_u=\Ee_u\otimes f^*Q_u.
\]
Since $\Ee_u$ has rank two,
\[
\det\widetilde \Ee_u \cong \Oo_X(-K_X-F_{p(u)})
 \otimes f^*(Q_u^{\otimes2}) \cong \Oo_X(-K_X-F_{p_0}),
\]
so $\left\{\widetilde \Ee_u\right\}_{u\in U}$ is a fixed-determinant algebraic family on $X$.

This family is not locally isotrivial. Indeed, for a general ruling fiber $F$ one has
\[
(\widetilde\Ee_u)_{|F} \cong \Oo_{\PP^1}(2)\oplus\Oo_{\PP^1}.
\]
The summand of degree two is the unique line subbundle of maximal degree. Its saturation in $\widetilde\Ee_u$ is therefore determined intrinsically by the isomorphism class of $\widetilde\Ee_u$. This saturation is
\[
\Ll_u:=\Tt_{X/C}\otimes f^*Q_u.
\]
The corresponding quotient line bundle is
\[
\Mm_u:=\widetilde\Ee_u/\Ll_u \cong f^*\Tt_C(-F_{p(u)})\otimes f^*Q_u.
\]
Consequently, the tensor quotient of these two line bundles,
\[
\Ll_u\otimes\Mm_u^{-1} \cong \Tt_{X/C}\otimes f^*K_C(F_{p(u)}),
\]
is determined by the isomorphism class of $\widetilde\Ee_u$. Suppose that $\widetilde\Ee_u\cong\widetilde\Ee_v$. The intrinsic characterization above implies
\[
\Ll_u\otimes\Mm_u^{-1} \cong \Ll_v\otimes\Mm_v^{-1}.
\]
Therefore, $f^*\Oo_C\bigl(p(u)-p(v)\bigr)\cong\Oo_X$. Since $f$ has a section, the pullback homomorphism $f^*\colon\Pic(C)\rightarrow\Pic(X)$ is injective. Hence
\[
\Oo_C\bigl(p(u)-p(v)\bigr)\cong\Oo_C.
\]
For a smooth projective curve of positive genus, this implies $p(u)=p(v)$. After shrinking the \'etale neighborhood $U$ analytically around $u_0$, the map $U\to C$ is injective, and hence $u=v$. Thus the family $\{\widetilde\Ee_u\}_{u\in U}$ is not locally isotrivial and gives a positive-dimensional fixed-determinant deformation of $\Ee_{p_0}$.

Now suppose that $D\equiv C_0+bF$ is a section. Since $f_{|D}\colon D\rightarrow C$ is an isomorphism, its ramification
scheme is empty, and \eqref{irr-ruled-relative-log} becomes
\[
0\to \Tt_{X/C}(-D) \to \Tt_X(-\log D) \to f^*\Tt_C \to0.
\]
Thus, in the notation of Remark~\ref{elem-observation}, we set $\Ll=\Tt_{X/C}(-D)$ and $\Mm=f^*\Tt_C$ to have $\Aa=\Ll\otimes\Mm^{-1} =\Tt_{X/C}(-D)\otimes f^*K_C$. Since $\Tt_{X/C}\equiv2C_0+eF$, we have $\Aa\equiv C_0+(e-b+2g-2)F$. A direct application of surface Riemann--Roch gives
\begin{equation}\label{irr-ruled-section-A-chi}
\chi(X,\Aa)=2g-2+e-2b.
\end{equation}
On the other hand, substituting $a=1$ into
\eqref{irr-ruled-chi} gives
\begin{equation}\label{irr-ruled-section-end-chi}
\chi\bigl(X,\mathcal{E}\!nd_0(\Tt_X(-\log D))\bigr)=g-1+e-2b.
\end{equation}
Since $\Aa_{|F}\cong\Oo_{\PP^1}(1)$, Serre duality and restriction to a general fiber give  $\mathrm{H}^2(X,\Aa)=0$. Consequently, \eqref{irr-ruled-section-A-chi} yields
\[
\mathrm{h}^0(X,\Aa)=2g-2+e-2b+\mathrm{h}^1(X,\Aa).
\]
Restricting the logarithmic tangent sequence to a general fiber $F$ gives an extension
\[
0\to\Oo_{\PP^1}(1) \to\Tt_X(-\log D)_{|F} \to\Oo_{\PP^1} \to0,
\]
which splits, and hence
\[
\Tt_X(-\log D)_{|F} \cong\Oo_{\PP^1}(1)\oplus\Oo_{\PP^1}.
\]
It follows that $\mathcal{E}\!nd_0(\Tt_X(-\log D))\otimes K_X$ restricts to $\Oo_{\PP^1}(-1) \oplus\Oo_{\PP^1}(-2) \oplus\Oo_{\PP^1}(-3)$ on $F$. Therefore,
\[
\mathrm{H}^0\bigl( X,\mathcal{E}\!nd_0(\Tt_X(-\log D))\otimes K_X\bigr)=0,
\]
and Serre duality gives $\mathrm{H}^2\bigl(X,\mathcal{E}\!nd_0(\Tt_X(-\log D))\bigr)=0$. Using the nilpotent sections described in Remark~\ref{elem-observation}, we obtain
\[
\mathrm{h}^0\bigl(X,\mathcal{E}\!nd_0(\Tt_X(-\log D))\bigr)\geq\mathrm{h}^0(X,\Aa)+\varepsilon,
\]
where $\varepsilon=1$ if the extension splits and $\varepsilon=0$ otherwise. Combining this inequality with \eqref{irr-ruled-section-A-chi} and \eqref{irr-ruled-section-end-chi}, we obtain
\[
\mathrm{h}^1\bigl(X,\mathcal{E}\!nd_0(\Tt_X(-\log D))\bigr)\geq g-1+\mathrm{h}^1(X,\Aa)+\varepsilon.
\]
If the extension is nonsplit, then its nonzero extension class lies in
\[
\mathrm{Ext}^1_X(\Mm,\Ll)\cong\mathrm{H}^1(X,\Aa),
\]
so $\mathrm{h}^1(X,\Aa)\geq1$. If the extension splits, then $\varepsilon=1$. Since $g\geq1$, the right-hand side is positive in both cases. Thus
\[
\mathrm{H}^1\bigl( X,\mathcal{E}\!nd_0(\Tt_X(-\log D))\bigr)\neq0.
\]
Moreover, since the corresponding second cohomology group vanishes,
\[
\mathrm{h}^1\bigl(X,\mathcal{E}\!nd_0(\Tt_X(-\log D))\bigr)-\mathrm{h}^2\bigl(X,\mathcal{E}\!nd_0(\Tt_X(-\log D))\bigr)>0.
\]
Hence Proposition~\ref{kuranishi} supplies a positive-dimensional fixed-determinant deformation in the section case.
\end{proof}

Let $D$ be a horizontal multisection of degree $a=D\cdot F\ge2$. The following Riemann--Hurwitz calculation applies for every value of the ruled-surface invariant $e$ and identifies exactly when the Euler characteristic can fail to be negative.

\begin{lemma}\label{irr-ruled-ramification}
Suppose that $a\geq2$, and let $R$ be the ramification divisor of $f_{\ |D}\colon D\rightarrow C$. Putting $\delta:=2b-ea$, we have
\[
\deg R=(a-1)\delta,\qquad D^2=a\delta,
\]
and
\begin{equation}\label{irr-ruled-chi-ramification}
\chi\bigl(\mathcal{E}\!nd_0(\Tt_X(-\log D))\bigr) =(5-4a)(g-1)+(2-3a)\delta.
\end{equation}
In particular, this Euler characteristic is negative unless $g=1$ and $f_{\ |D}$ is \'etale. In the exceptional case, $D^2=0$ and the Euler characteristic is zero.
\end{lemma}

\begin{proof}
The adjunction formula and Riemann--Hurwitz formula give, respectively,
\[
2g(D)-2=D\cdot(D+K_X) \quad\hbox{and}\quad 2g(D)-2=a(2g-2)+\deg R.
\]
Subtracting the second right-hand side from the first, and using the intersection calculations above, gives
\[
\deg R=D^2+K_X\cdot D-a(2g-2)=(a-1)(2b-ea).
\]
The identity $D^2=a(2b-ea)$ proves the first assertion. Substituting $b=(\delta+ea)/2$ in \eqref{irr-ruled-chi} cancels all terms involving $e$ and yields \eqref{irr-ruled-chi-ramification}. Since $a\geq2$, both coefficients $5-4a$ and $2-3a$ are negative. Moreover $\delta\geq0$ because it is $\deg R/(a-1)$. Equality can therefore
occur only when $g=1$ and $\deg R=0$; the converse and the assertion about $D^2$ are immediate.
\end{proof}

It remains to treat precisely the equality case of Lemma~\ref{irr-ruled-ramification}. We use ``foliation'' below only as shorthand for a line subbundle of $\Tt_X$ which maps isomorphically to $f^*\Tt_C$; equivalently, it is a splitting of the relative tangent sequence.

\begin{proposition}\label{irr-ruled-etale-multisection}
Let $C$ be elliptic and let $D\subset X=\PP_C(\Vv)$ be a smooth
irreducible \'etale multisection of degree $a\geq2$. Then the relative logarithmic sequence \eqref{irr-ruled-relative-log} splits:
\[
\Tt_X(-\log D)\cong  \Tt_{X/C}(-D)\oplus\Oo_X.
\]
In particular, we have $\mathrm{H}^1\bigl(X,\mathcal{E}\!nd_0(\Tt_X(-\log D))\bigr)\neq0$. Moreover, the logarithmic tangent bundle has a positive-dimensional fixed-determinant deformation.
\end{proposition}

\begin{proof}
By Lemma~\ref{irr-ruled-ramification}, one has $D^2=0$. Let $u\colon C'\rightarrow C$ be a finite \'etale cover dominating the Galois closure of $D\to C$, and form the Cartesian diagram
\[
\begin{tikzcd}
X'=X\times_C C' \arrow[r,"q"] \arrow[d,"f'"]&
X \arrow[d,"f"]\\
C' \arrow[r,"u"]&C.
\end{tikzcd}
\]
Then $q$ is finite and \'etale, and the inverse image of $D$ is a disjoint union
\[
D'=q^{-1}(D)=S_1+\cdots+S_a
\]
of sections of $f'$. Since the sections are disjoint and conjugate and $(D')^2=(\deg u)D^2=0$, one has $S_i^2=0$ for every $i$.

\begin{claim}\label{upstair-split}
The pullback of the relative logarithmic sequence
\begin{equation}\label{irr-ruled-pulled-log-sequence}
0\to  \Tt_{X'/C'}(-D') \to \Tt_{X'}(-\log D') \to \Oo_{X'} \to 0
\end{equation}
splits.
\end{claim}
\begin{claimproof}
Suppose first that $a=2$. The two disjoint sections reduce the structure group of the $\PP^1$-bundle, and, after tensoring the underlying rank-two bundle by a line bundle, we may write
\[
X'\cong\PP_{C'}(\Oo_{C'}\oplus\Mm),
\]
where $S_1$ and $S_2$ are the two distinguished sections. Their self-intersection zero implies $\deg\Mm=0$.  Choose an open covering $\{U_i\}$ of $C'$ on which $\Mm$ is trivial, and let $g_{ij}\in\Oo_{C'}^*(U_i\cap U_j)$ be its transition functions. If $z_i$ denotes the corresponding affine fiber coordinate, then on overlaps $z_j=g_{ij}z_i$. The local product splittings of
\eqref{irr-ruled-pulled-log-sequence} differ on $U_i\cap U_j$ by
\[
d\log(g_{ij})\otimes z_i\frac{\partial}{\partial z_i}.
\]
Thus the extension class of \eqref{irr-ruled-pulled-log-sequence} is represented by the Čech cocycle $\left\{d\log(g_{ij})\right\}$, which is the Atiyah class, or equivalently the first Chern class, of $\Mm$. On a smooth projective curve this class is determined by $\deg\Mm$. Since $\deg\Mm=0$, the extension class vanishes, and
\eqref{irr-ruled-pulled-log-sequence} splits.

Now suppose that $a\geq3$. Choose three of the disjoint sections and order them. There is a unique fiberwise projective transformation sending these three sections to $0$, $1$, and $\infty$. These fiberwise transformations glue and give a trivialization $X'\cong C'\times\PP^1$. Every remaining section is constant in this trivialization. Indeed, such a section is the graph of a morphism
\[
C'\to \PP^1\setminus\{0,1,\infty\},
\]
which must be constant because $C'$ is complete. Hence $D'=C'\times A$ for a finite subset $A\subset\PP^1$. Therefore
\[
\Tt_{X'}(-\log D') \cong p_{\PP^1}^*\Tt_{\PP^1}(-\log A) \oplus p_{C'}^*\Tt_{C'}.
\]
Since $C'$ is elliptic, $p_{C'}^*\Tt_{C'}\cong\Oo_{X'}$, so \eqref{irr-ruled-pulled-log-sequence} splits in this case as well.
\end{claimproof}

Let
\[
\xi\in \Ext_X^1\bigl(\Oo_X,\Tt_{X/C}(-D)\bigr)=\mathrm{H}^1\bigl(X,\Tt_{X/C}(-D)\bigr)
\]
be the extension class of \eqref{irr-ruled-relative-log}. The splitting upstairs shows that $q^*\xi=0$. Since $q$ is finite and \'etale, the normalized trace gives a splitting
\[
\Oo_X\to q_*\Oo_{X'} \to \Oo_X
\]
of the natural inclusion. By the projection formula,
\[
\mathrm{H}^1\bigl(X,\Tt_{X/C}(-D)\bigr) \to \mathrm{H}^1\bigl( X',q^*\Tt_{X/C}(-D)\bigr)
\]
is therefore injective. Hence $q^*\xi=0$ implies $\xi=0$. Thus the relative logarithmic sequence splits:
\[
\Tt_X(-\log D) \cong \Tt_{X/C}(-D)\oplus\Oo_X.
\]
The diagonal endomorphism with eigenvalues $1$ and $-1$ is consequently a nonzero section of
\[
\mathcal E\!nd_0(\Tt_X(-\log D)).
\]
Lemma~\ref{irr-ruled-ramification} gives $\chi\bigl(\mathcal E\!nd_0(\Tt_X(-\log D))\bigr)=0$, and therefore,
\[
\begin{aligned}
&\mathrm{h}^1\bigl(X,\mathcal E\!nd_0(\Tt_X(-\log D))\bigr)-\mathrm{h}^2\bigl(X,\mathcal E\!nd_0(\Tt_X(-\log D)) \bigr)\\
&\qquad =\mathrm{h}^0\bigl(X,\mathcal E\!nd_0(\Tt_X(-\log D))\bigr)>0.\end{aligned}
\]
Proposition~\ref{kuranishi} now gives a positive-dimensional fixed-determinant deformation.
\end{proof}

Every nonempty smooth irreducible divisor $D\subset X$ falls into one of three mutually exclusive classes according to $a=D\cdot F$. If $a=0$, then $D$ is a fiber of $f$. If $a=1$, then $f_{|D}\colon D\rightarrow C$ is an isomorphism and $D$ is a section. If $a\geq2$, then $D$ is a horizontal multisection.

The preceding arguments impose no sign condition on the ruled-surface invariant $e$. Proposition~\ref{irr-ruled-fibers-sections} treats fibers and sections for every $e$, while Lemma~\ref{irr-ruled-ramification} and Proposition~\ref{irr-ruled-etale-multisection} treat all multisections of degree at least two. Hence all ruled surfaces with negative invariant are included; in particular, this covers the stable ruled surface with $g(C)=1$ and $e=-1$.

\begin{theorem}\label{irr-ruled-complete}
Let $X=\PP_C(\Vv)$ be a minimal ruled surface over a curve of genus $g(C)\geq1$, and let $D\subset X$ be a nonempty smooth irreducible divisor. Then we have
\[
\mathrm{H}^1\bigl(X,\mathcal{E}\!nd_0(\Tt_X(-\log D))\bigr)\neq0.
\]
Moreover, $\Tt_X(-\log D)$ has a positive-dimensional fixed-determinant deformation.
\end{theorem}

\begin{proof}
Put $a=D\cdot F$. If $a=0$, then $f_{|D}$ is constant, so $D$ is contained in a fiber of $f$. Since every fiber of a ruled surface is an irreducible copy of $\PP^1$, it follows that $D$ is a fiber. If $a=1$, then $f_{|D}\colon D\rightarrow C$ is a finite morphism of degree one between smooth projective curves. It is therefore an isomorphism, and $D$ is a section. The fiber and section cases are covered by Proposition~\ref{irr-ruled-fibers-sections}. Now suppose that $a\geq2$. By Lemma~\ref{irr-ruled-ramification}, either
\[
\chi\bigl(X,\mathcal{E}\!nd_0(\Tt_X(-\log D))\bigr)<0,
\]
in which case Corollary~\ref{numerical} gives a positive-dimensional fixed-determinant deformation, or $C$ is elliptic and $f_{|D}\colon D\rightarrow C$ is \'etale. In the latter case, Proposition~\ref{irr-ruled-etale-multisection} applies.

These arguments impose no sign condition on $e$, so they also cover all ruled surfaces with negative invariant.
\end{proof}


\section{Surfaces of Kodaira dimension zero}

We can completely treat minimal surfaces of Kodaira dimension zero with a smooth irreducible boundary.

\begin{theorem}\label{kappa-zero}
Let $X$ be a minimal smooth projective surface with $\kappa(X)=0$. For $D\subset X$ either empty or a nonempty smooth irreducible divisor, we have
\[
\mathrm{H}^1\bigl(X,\mathcal{E}\!nd_0(\Tt_X(-\log D))\bigr)\neq0.
\]
Thus, for every such pair, the logarithmic tangent bundle admits a nontrivial fixed-determinant first-order deformation. If $D$ is nonempty, the logarithmic tangent bundle has a positive-dimensional fixed-determinant deformation.
\end{theorem}

\begin{proof}
The minimal projective surfaces of Kodaira dimension zero are K3, Enriques, abelian, and bielliptic surfaces. Since $K_X$ is numerically trivial, Theorem~\ref{RR} becomes
\begin{equation}\label{kappa-zero-chi}
\chi\bigl(\mathcal{E}\!nd_0(\Tt_X(-\log D))\bigr)=-\frac{15}{4}c_2(X)-3D^2.
\end{equation}
Suppose first that $X$ is a K3 surface so that $c_2(X)=24$. The adjunction formula gives $D^2=2g(D)-2\geq-2$, and hence
\[
\chi\bigl(\mathcal{E}\!nd_0(\Tt_X(-\log D))\bigr) =-90-3D^2\leq-84<0.
\]
For an Enriques surface, $c_2(X)=12$, and the adjunction formula again gives $D^2\geq-2$. Therefore
\[
\chi\bigl(\mathcal{E}\!nd_0(\Tt_X(-\log D))\bigr)=-45-3D^2\leq-39<0.
\]
Corollary~\ref{numerical} proves both assertions in these cases.

Now suppose that $X$ is abelian or bielliptic. Then $c_2(X)=0$, and there are no rational curves so that the adjunction formula gives $D^2=2g(D)-2\geq0$. If $D^2>0$, \eqref{kappa-zero-chi} is negative, so we again apply Corollary~\ref{numerical}.

It remains to consider $D^2=0$. Then $D$ is an elliptic curve. On an abelian surface, $D$ is a translate of an elliptic subtorus. After translating $D$ to a subgroup, let $V_1\subset \mathrm{H}^0(X,\Tt_X)$ be its invariant tangent direction and choose a complementary invariant direction $V_2$. Multiplication of $V_2$ by a defining section of $D$ then gives
\[
\Tt_X(-\log D)\cong \Oo_X\oplus\Oo_X(-D).
\]
In particular, the diagonal endomorphism with eigenvalues $1$ and $-1$ is a nonzero section of $\mathcal{E}\!nd_0(\Tt_X(-\log D))$.

Now suppose that $X$ is bielliptic, and let $q\colon A\rightarrow X$ be its finite \'etale abelian cover. Put $\Delta=q^{-1}(D)$. Since $\Delta^2=0$, its irreducible components are pairwise disjoint translates of a common elliptic subtorus $B\subset A$. Indeed, two elliptic curves on an abelian surface having distinct tangent directions have positive intersection number.

The deck group $G$ permutes the components of $\Delta$ and therefore preserves the invariant tangent direction $V_B\subset \mathrm{H}^0(A,\Tt_A)$ of $B$. Since $G$ is finite, $V_B$ admits a $G$-invariant complementary subspace $V'$ in $\mathrm{H}^0(A,\Tt_A)$. Consequently, the logarithmic splitting
\[
\Tt_A(-\log\Delta) \cong (V_B\otimes\Oo_A) \oplus (V'\otimes\Oo_A(-\Delta))
\]
is $G$-equivariant. It therefore descends to a splitting
\[
\Tt_X(-\log D)\cong\Ll_1\oplus\Ll_2
\]
into two line bundles. Hence the diagonal endomorphism with eigenvalues $1$ and $-1$ gives a nonzero section of
\[
\mathcal{E}\!nd_0(\Tt_X(-\log D)).
\]
Since \eqref{kappa-zero-chi} is zero, Riemann--Roch gives
\[
\mathrm{h}^1(\mathcal{E}\!nd_0(\Tt_X(-\log D)))=\mathrm{h}^0(\mathcal{E}\!nd_0(\Tt_X(-\log D)))+\mathrm{h}^2(\mathcal{E}\!nd_0(\Tt_X(-\log D)))>0.
\]

Finally, when $D$ is empty, the same conclusion follows directly. On a K3 or Enriques surface the Euler characteristic is negative. On an abelian surface the tangent bundle is trivial. On a bielliptic surface the two invariant tangent directions on the \'etale abelian cover descend as line subbundles, and the diagonal trace-free endomorphism together with $q(X)=1$ gives nonzero first cohomology.
\end{proof}

\begin{remark}
Irreducibility is used only for the estimate $D^2\geq-2$ on K3 and Enriques surfaces. A smooth disconnected divisor may contain many disjoint $(-2)$-curves, so equation \eqref{kappa-zero-chi} need not be negative. That disconnected-boundary case requires a separate analysis of the configuration and is not included in Theorem~\ref{kappa-zero}.
\end{remark}


\section{Surfaces of Kodaira dimension one}

Let $X$ be a minimal projective surface with $\kappa(X)=1$. Its Iitaka fibration is an elliptic fibration $f\colon X\rightarrow B$, and $K_X$ is nef with $K_X^2=0$. The logarithmic Riemann--Roch formula therefore reads
\begin{equation}\label{kappa-one-chi}
\chi\bigl(\mathcal{E}\!nd_0(\Tt_X(-\log D))\bigr)=-\frac{15}{4}c_2(X)-2K_X\cdot D-3D^2.
\end{equation}
When $D$ is empty, nonrigidity follows already from Peternell's surface classification~\cite[Theorem~1.4]{Peternell}, since a surface of Kodaira dimension one is neither $\PP^2$ nor a ball quotient.

\begin{theorem}\label{kappa-one-nef}
Let $X$ be a minimal smooth projective surface with $\kappa(X)=1$, and let $D\subset X$ be a nonempty smooth irreducible nef divisor. Then
\[
\mathrm{H}^1\bigl(X,\mathcal{E}\!nd_0(\Tt_X(-\log D))\bigr)\neq0.
\]
\end{theorem}

\begin{proof}
Note that the three numbers $c_2(X),\ K_X\cdot D,\ D^2$ are nonnegative. If one is positive, equation \eqref{kappa-one-chi} is negative and Corollary~\ref{numerical} applies.

Assume that $c_2(X)=K_X\cdot D=D^2=0$. Letting $F$ denote the numerical class of a general fiber of $f$, the canonical bundle formula gives
\[
K_X\equiv \left(2g(B)-2+\chi(\Oo_X) +\sum_i\left(1-\frac1{m_i}\right)\right)F,
\]
where $m_iF_i$ are the multiple fibers over $b_i\in B$. Since $\kappa(X)=1$, the coefficient of $F$ is strictly positive. Consequently, $K_X\cdot D=0$ implies $F\cdot D=0$, so $D$ is vertical. Since $c_2(X)=0$, the elliptic fibration has no fiber whose reduced support has positive topological Euler characteristic. Hence every fiber is a smooth elliptic curve, possibly occurring with multiplicity. Since $D$ is smooth, irreducible, and vertical, it is the reduced support of one of these fibers and is therefore an elliptic curve.

The fibration is consequently isotrivial. Since every singular fiber is a multiple of a smooth elliptic curve, $f$ is an elliptic quasi-bundle. By the product-quotient description of elliptic quasi-bundles \cite[Section~2]{SerranoIsotrivial}, there exist an elliptic curve $A$, a smooth projective curve $C'$, and a finite étale morphism
\[
q\colon A\times C'\to X
\]
compatible with the respective fibrations. For suitable points
$c_1,\ldots,c_r\in C'$, one has
\[
q^{-1}(D)=\coprod_{j=1}^r A\times\{c_j\}.
\]
Writing $\Delta=q^{-1}(D)$, the logarithmic tangent bundle of the product pair splits as
\[
\Tt_{A\times C'}(-\log\Delta) \cong  p_A^*\Tt_A\oplus p_{C'}^*\ \Tt_{C'}\bigl(-c_1-\cdots-c_r\bigr).
\]
The two summands are the vertical and horizontal tangent directions. The product cover may be chosen so that $X$ is a quotient of $A\times C'$ by a finite group acting diagonally. Hence both summands are invariant under the deck group and descend to line bundles $\Ll$ and $\Mm$ on $X$. Consequently,
\[
\Tt_X(-\log D)\cong \Ll\oplus \Mm.
\]
The endomorphism acting as the identity on $\Ll$ and as minus the identity on $\Mm$ is a nonzero trace-free endomorphism. Therefore
\[
\mathrm{h}^0\bigl(X,\mathcal{E}\!nd_0(\Tt_X(-\log D))\bigr)\geq1.
\]
In the present case, we have $c_2(X)=K_X\cdot D=D^2=0$, so \eqref{kappa-one-chi} gives $\chi\bigl(X,\mathcal{E}\!nd_0(\Tt_X(-\log D))\bigr)=0$. Riemann--Roch therefore yields
\[
\mathrm{h}^1\bigl(X,\mathcal{E}\!nd_0(\Tt_X(-\log D)))\bigr)=\mathrm{h}^0\bigl(X,\mathcal{E}\!nd_0(\Tt_X(-\log D))\bigr)+\mathrm{h}^2\bigl(X,\mathcal{E}\!nd_0(\Tt_X(-\log D))\bigr)>0. \qedhere
\]
\end{proof}

The assumption that $D$ is nef in Theorem \ref{kappa-one-nef} isolates the negative-curve case: if $D$ is irreducible and not nef, then $D^2<0$. If $\Tt_X(-\log D)$ is infinitesimally rigid with fixed determinant, then $\mathrm{H}^1\bigl(X,\mathcal{E}\!nd_0(\Tt_X(-\log D))\bigr)=0$ and so $\chi\bigl(X,\mathcal{E}\!nd_0(\Tt_X(-\log D))\bigr)\geq 0$. By \eqref{kappa-one-chi}, one gets 
\begin{equation}\label{kappa-one-negative-necessary}
-3D^2\geq\frac{15}{4}c_2(X)+2K_X\cdot D.
\end{equation}
Thus only sufficiently negative curves can survive the test in Corollary \ref{numerical}. 

\begin{proposition}\label{kappa-one-negative-cases}
Let $X$ be a minimal smooth projective surface with $\kappa(X)=1$, and let $f\colon X\rightarrow B$ be its elliptic Iitaka fibration. If $D\subset X$ is a nonempty smooth irreducible curve, then the following hold:
\begin{enumerate}[label=\textup{(\arabic*)}]
\item if $D$ is a rational component of a singular fiber, then $\mathrm{H}^1(X,\mathcal{E}\!nd_0(\Tt_X(-\log D)))\neq0$;
\item if $D$ is a section of $f$, then $\mathrm{H}^1(X,\mathcal{E}\!nd_0(\Tt_X(-\log D)))\neq0$.
\end{enumerate}
\end{proposition}

\begin{proof}
For a smooth rational component $D$ of a relatively minimal elliptic fiber, $K_X\cdot D=0$ and $D^2=-2$. Moreover, the existence of a singular fiber gives $c_2(X)>0$, and $c_2(X)=12\chi(\Oo_X)\geq12$. Hence
\[
\chi(\mathcal{E}\!nd_0(\Tt_X(-\log D)))=-\frac{15}{4}c_2(X)+6<0.
\]
If $D$ is a section, then
\[
D^2=-\chi(\Oo_X),\qquad K_X\cdot D=2g(B)-2+\chi(\Oo_X);
\]
an elliptic fibration with a section has no multiple fibers in its scheme-theoretic fiber class. Substitution yields
\[
\chi(\mathcal{E}\!nd_0(\Tt_X(-\log D)))=-44\chi(\Oo_X)-4g(B)+4<0,
\]
because $\kappa (X)=1$ gives $2g(B)-2+\chi(\Oo_X)>0$. In both cases, we can use Corollary \ref{numerical}.
\end{proof}

We next record the ramification formulas for a negative horizontal multisection. Write
\begin{equation}\label{notation-num}
\chi_0=\chi(\Oo_X), \qquad s=\sum_i\left(1-\frac1{m_i}\right),
\end{equation}
where $m_i$ are the multiplicities of the multiple fibers. The canonical bundle formula gives
\begin{equation}\label{can-bundle-form}
K_X\equiv \bigl(2g(B)-2+\chi_0+s\bigr)F
\end{equation}
numerically, where $F$ is the numerical fiber class.

\begin{proposition}\label{kappa-one-multisection-reduction}
Let $D$ be a smooth irreducible horizontal multisection of degree $n=D\cdot F$, and let $R$ be the total ramification degree of $f_{\ |D}\colon D\rightarrow B$. Then we have
\[
\begin{cases}
~~K_X\cdot D =n\bigl(2g(B)-2+\chi_0+s\bigr); \\
~~D^2=R-n(\chi_0+s); \\
~~\chi\bigl(\mathcal{E}\!nd_0(\Tt_X(-\log D))\bigr)
=(n-45)\chi_0-4n(g(B)-1)+ns-3R.
\end{cases}
\]
Moreover, we have $R\geq ns$. Consequently, genuine rigidity is numerically possible only if
\begin{equation}\label{multi-necessary-general}
n\bigl(\chi_0-4(g(B)-1)-2s\bigr) \geq45\chi_0.
\end{equation}
In particular, every horizontal multisection is genuinely nonrigid if
\begin{equation}\label{multi-easy-range}
\chi_0\leq4(g(B)-1)+2s.
\end{equation}
\end{proposition}

\begin{proof}
The first equality is obtained by taking the intersection of \eqref{can-bundle-form} with $D$. By the Riemann--Hurwitz formula and the adjunction formula, we obtain
\[
n(2g(B)-2)+R= 2g(D)-2=(K_X+D)\cdot D,
\]
yielding the second. Since $c_2(X)=12\chi_0$, substitution in \eqref{kappa-one-chi} yields the third.

Fix a multiple fiber $m_iF_i=f^{-1}(b_i)$, where $F_i$ is the primitive fiber divisor and $m_i\geq2$. Set $\nu=f_{\ |D}$ and write $a_p:=I_p(D,F_i)$ for every point $p\in D$ with $\nu(p)=b_i$. Since a parameter $t$ at $b_i$ pulls locally back to an equation of $m_iF_i$, we obtain 
\[
e_p:=\mathrm{ord}_p(\nu^*t)=m_i a_p.
\]
Thus $e_p$ is the ramification index of $\nu$ at $p$. Since $\nu$ has degree $n$, $\sum_{p\in\nu^{-1}(b_i)}e_p=n$, and therefore
\[
\sum_{p\in\nu^{-1}(b_i)}a_p=\frac{n}{m_i}.
\]
In particular, the number of distinct points lying over $b_i$ satisfies $\#\nu^{-1}(b_i)\leq\frac{n}{m_i}$, and hence we obtain
\[
\sum_{p\in\nu^{-1}(b_i)}(e_p-1) =n-\#\nu^{-1}(b_i)\geq n-\frac{n}{m_i}=n\left(1-\frac1{m_i}\right).
\]
Summing over all multiple fibers gives $R\geq ns$. 

Suppose now that $\Tt_X(-\log D)$ is infinitesimally rigid with fixed determinant. This implies $\mathrm{H}^1\bigl(X,\mathcal{E}\!nd_0(\Tt_X(-\log D))\bigr)=0$, and hence $\chi\bigl(X,\mathcal{E}\!nd_0(\Tt_X(-\log D))\bigr) \geq0$. On the other hand, the third equality and $R\geq ns$ give
\begin{align*}
\chi\bigl(X,\mathcal{E}\!nd_0(\Tt_X(-\log D))\bigr) &=(n-45)\chi_0-4n(g(B)-1)+ns-3R\\
&\leq(n-45)\chi_0-4n(g(B)-1)-2ns\\
&=n\bigl(\chi_0-4(g(B)-1)-2s\bigr)-45\chi_0.
\end{align*}
The right-hand side must therefore be nonnegative, and this gives \eqref{multi-necessary-general}. If this inequality fails, then the preceding upper bound is strictly negative, so the Euler characteristic itself is negative. Corollary~\ref{numerical} then produces a positive-dimensional fixed-determinant deformation.
\end{proof}

\begin{theorem}\label{kappa-one-negative-multisection}
Let $X$ be a minimal smooth projective surface with $\kappa(X)=1$, and let $D\subset X$ be a smooth irreducible horizontal multisection with $D^2<0$. Then
\[
\mathrm{H}^1\bigl(X,\mathcal{E}\!nd_0(\Tt_X(-\log D))\bigr)\neq0.
\]
In fact, the logarithmic tangent bundle has a genuine positive-dimensional deformation.
\end{theorem}

\begin{proof}
We apply the orbibundle Miyaoka--Yau--Sakai inequality \cite[Theorem~1.3(i)]{MiyaokaOrbibundle} to the reduced irreducible curve $D$. Its hypotheses are satisfied: $X$ is a smooth projective surface of nonnegative Kodaira dimension, and $D$ is smooth and irreducible, so in particular its arithmetic genus equals its geometric genus. The inequality states that, for every $\alpha\in[0,1]$,
\[
\frac{\alpha^2}{2}\bigl(D^2+3K_X\cdot D-6g(D)+6\bigr)-2\alpha\bigl(K_X\cdot D-3g(D)+3\bigr)+3c_2(X)-K_X^2 \geq0.
\]
We now take $\alpha=1$ and set
\[
\nu_D:=-D^2>0,\qquad k_D:=K_X\cdot D.
\]
Since $D$ is horizontal and the numerical coefficient of the fiber class in the canonical bundle formula is strictly positive, one has $k_D>0$. The preceding inequality becomes
\[
\frac12(-\nu_D+3k_D-6g(D)+6)-2(k_D-3g(D)+3) +3c_2(X)-K_X^2\geq0.
\]
The adjunction formula gives $2g(D)-2=k_D-\nu_D$. Since $K_X^2=0$ and $c_2(X)=12\chi_0$, the preceding inequality simplifies to $k_D-2\nu_D+36\chi_0\geq0$. Equivalently,
\begin{equation}\label{kappa-one-miyaoka-bound}
\nu_D\leq18\chi_0+\frac{k_D}{2}.
\end{equation}
Suppose, toward a contradiction, that the Euler characteristic $\chi\bigl(\mathcal{E}\!nd_0(\Tt_X(-\log D))\bigr)$ is nonnegative. Then \eqref{kappa-one-chi} gives
\begin{equation}\label{kappa-one-rigidity-lower-q}
\nu_D\geq15\chi_0+\frac{2k_D}{3}.
\end{equation}
Combining \eqref{kappa-one-miyaoka-bound} and \eqref{kappa-one-rigidity-lower-q} yields $k_D\leq18\chi_0$. In particular, one obtains $\chi_0>0$.

If $g(B)\geq1$, equations in Proposition \ref{kappa-one-multisection-reduction} and $R\geq0$ give
\[
\nu_D=n(\chi_0+s)-R\leq n(\chi_0+s)\leq k_D.
\]
Combining this with \eqref{kappa-one-rigidity-lower-q} gives $k_D\geq45\chi_0$, which results in $45\chi_0 \le 18 \chi_0$, a contradiction. If $B\cong \PP^1$, the Riemann--Hurwitz formula gives $R\geq2n-2$, and hence
\[
\nu_D=n(\chi_0+s)-R \leq n(\chi_0+s-2)+2=k_D+2.
\]
Equation \eqref{kappa-one-rigidity-lower-q} now gives $45\chi_0-6\le k_D \le 18\chi_0$. Since $\chi_0$ is a positive integer, this is impossible. Thus we get $\chi\bigl(\mathcal{E}\!nd_0(\Tt_X(-\log D))\bigr)<0$, and Corollary~\ref{numerical} applies.
\end{proof}

\begin{corollary}\label{kappa-one-classification}
Let $X$ be a minimal smooth projective surface with $\kappa(X)=1$, and let $D\subset X$ be a nonempty smooth irreducible divisor. Then
\[
\mathrm{H}^1\bigl(X,\mathcal{E}\!nd_0(\Tt_X(-\log D))\bigr)\neq0.
\]
Consequently there is no genuinely rigid pair of this type.
\end{corollary}

\begin{proof}
An irreducible curve on the elliptic surface is either vertical or horizontal. If it is nef, Theorem~\ref{kappa-one-nef} gives the required nonvanishing. Suppose it is not nef. Then $D^2<0$. A vertical irreducible curve lies in a fiber. A smooth elliptic reduced multiple fiber has square zero, while every smooth component of a reducible Kodaira fiber is rational of square $-2$; thus the non-nef vertical case is covered by Proposition~\ref{kappa-one-negative-cases}(1). Finally, a horizontal curve is a multisection of positive degree. Its negative-square case is precisely Theorem~\ref{kappa-one-negative-multisection}. These mutually exclusive cases exhaust all smooth irreducible divisors $D$.
\end{proof}


\section{Surfaces of general type}

Let $X$ be a minimal surface of general type. Then $K_X$ is nef and big, and the Miyaoka--Yau defect is nonnegative:
\[
\delta(X)=3c_2(X)-K_X^2 \ge 0.
\]
When $D$ is empty, Peternell's theorem in \cite{Peternell} gives an exceptional possibility: $\Tt_X$ is genuinely rigid precisely for the indicated ball quotients satisfying the additional symmetric-differential vanishing. But the presence of any nonempty smooth irreducible boundary precludes fixed-determinant infinitesimal rigidity of the logarithmic tangent bundle.

\begin{theorem}\label{kappa-two}
Let $X$ be a minimal smooth projective surface of general type and let $D\subset X$ be a nonempty smooth irreducible curve. Then
\[
\mathrm{H}^1\bigl(X,\mathcal{E}\!nd_0(\Tt_X(-\log D))\bigr)\neq0.
\]
Consequently, no such pair is genuinely rigid.
\end{theorem}

\begin{proof}
Assume first that $g(D)\ge 1$. For every irreducible curve $C\neq D$, one has $(K_X+D)\cdot C\geq0$, because $K_X$ is nef. The adjunction formula gives
\[
(K_X+D)\cdot D=2g(D)-2\geq0.
\]
Thus $K_X+D$ is nef; it is big because $K_X$ is big. The logarithmic Bogomolov--Miyaoka--Yau (BMY) inequality gives
\begin{equation}\label{log-BMY-general-type}
(K_X+D)^2\leq3\bar c_2(X,D).
\end{equation}
Moreover, the adjunction formula gives
\begin{equation}\label{log-c2-adjunction}
\bar c_2(X,D)=c_2(X)+K_X\cdot D+D^2 =c_2(X)+2g(D)-2.
\end{equation}
By Theorem~\ref{RR} and \eqref{log-BMY-general-type},
\begin{align*}
\chi\bigl(\mathcal{E}\!nd_0(\Tt_X(-\log D))\bigr)&\leq3\chi(\Oo_X)-\bar c_2(X,D)\\
&=-\frac14\bigl(\delta(X)+8g(D)-8\bigr).
\end{align*}
If $g(D)\geq2$, this is strictly negative. If $g(D)=1$, the right-hand side is strictly negative unless $\delta(X)=0$, equivalently, $K_X^2=3c_2(X)$. In this equality case, the equality characterization in the Miyaoka--Yau theorem \cite[Chapter~VII]{BHPV} implies that $X$ is a compact quotient of the complex two-ball $\BB^2$. In particular, $X$ contains no elliptic curve. Indeed, if a nonconstant morphism $\varphi\colon E\rightarrow X$ existed from an elliptic curve $E$, then its composition with the universal covering map $\CC\rightarrow E$ would lift to a holomorphic map
\[
\widetilde\varphi\colon\CC\rightarrow\mathbb B^2.
\]
Since $\BB^2\subset\CC^2$ is bounded, both coordinate functions of $\widetilde\varphi$ are bounded entire functions and are therefore constant. Hence $\varphi$ is constant, a contradiction. Thus the equality case cannot occur in the presence of the elliptic curve $D$, and consequently
\[
\chi\bigl(X,\mathcal E\!nd_0(\Tt_X(-\log D))\bigr)<0.
\]
Corollary~\ref{numerical} now applies.

Now assume that $g(D)=0$, i.e. $D$ is a smooth rational curve. Then we have $(K_X+D)\cdot D=-2$, so $K_X+D$ is not nef and logarithmic BMY inequality cannot be applied directly. The adjunction formula and Theorem~\ref{RR} give
\begin{equation}\label{kappa-two-rational-chi}
\chi\bigl(\mathcal{E}\!nd_0(\Tt_X(-\log D))\bigr)=-\frac54\delta(X)+4-D^2.
\end{equation}
Since $K_X$ is nef, adjunction implies $D^2\leq-2$. A necessary numerical condition for rigidity is therefore
\begin{equation}\label{kappa-two-rational-necessary}
-D^2+4\geq\frac54\delta(X).
\end{equation}
Miyaoka's sharp inequality for a smooth rational curve on a minimal nonruled surface gives, with $m=-D^2$,
\begin{equation}\label{miyaoka-rational-curve}
K_X\cdot D=m-2\leq \left\lfloor 3c_2(X)-K_X^2-4-\frac1m \right\rfloor =\delta(X)-5;
\end{equation}
see \cite[Remark A]{MiyaokaOrbibundle}. In particular, $m\leq\delta(X)-3$. Noether's formula shows that
\[
\delta(X)=3c_2(X)-K_X^2=4\bigl(c_2(X)-3\chi(\Oo_X)\bigr)
\]
is divisible by four. Inequality \eqref{miyaoka-rational-curve} and $m\geq2$ imply $\delta(X)\geq m+3\geq5$, hence in fact $\delta(X)\geq8$. Combining $m\leq\delta(X)-3$ with \eqref{kappa-two-rational-chi} gives
\[
\chi\bigl(\mathcal{E}\!nd_0(\Tt_X(-\log D))\bigr) \leq-\frac54\delta(X)+4+\delta(X)-3 =1-\frac14\delta(X)<0.
\]
Corollary~\ref{numerical} now applies.
\end{proof}

\begin{proof}[Proof of Theorem \ref{minimal-surface-pairs}]
If $\kappa(X)=-\infty$, then $X$ is either rational or ruled over a curve of positive genus. The rational minimal surfaces are $\PP^2$ and $\FF_e$ with $e=0$ or $e\geq2$. The irrational ruled case is Theorem~\ref{irr-ruled-complete}. Minimal surfaces of Kodaira dimensions $0$, $1$, and $2$ are covered, respectively, by Theorem~\ref{kappa-zero}, Corollary~\ref{kappa-one-classification}, and Theorem~\ref{kappa-two}. These cases exhaust the Enriques--Kodaira classification.
\end{proof}


\section{Blow-ups}\label{blow-up}

Let $\pi\colon Y\rightarrow X$ be the blow-up of a point $p$, let $E$ be its exceptional curve, and let $D_Y$ be the reduced strict transform of $D$. The pair $(Y,D_Y)$ is again smooth.

\begin{proposition}\label{blowup-chern}
The logarithmic Chern numbers transform according to the following table:
\[
\renewcommand{\arraystretch}{1.4}
\setlength{\arraycolsep}{14pt}
\begin{array}{c|c|c}
\hline
&\bar c_1^2(Y,D_Y) &\bar c_2(Y,D_Y) \\ \hline \hline
p\notin D & \bar c_1^2(X,D)-1 &\bar c_2(X,D)+1 \\ \hline
p\in D&\bar c_1^2(X,D)&\bar c_2(X,D)+1 \\ \hline
\end{array}
\]
\end{proposition}

\begin{proof}
If $p\notin D$, then $K_Y+D_Y=\pi^*(K_X+D)+E$. If $p\in D$, then $D_Y=\pi^*D-E$ and
\[
K_Y+D_Y=\pi^*(K_X+D).
\]
Because $\pi^*(K_X+D)\cdot E=0$ and $E^2=-1$, squaring gives
\[
(K_Y+D_Y)^2=(K_X+D)^2-1
\]
off the boundary, and $(K_Y+D_Y)^2=(K_X+D)^2$ on the boundary, explaining $\bar c_1^2(Y,D_Y)$.

For $\bar c_2(Y,D_Y)$, recall $c_2(Y)=c_2(X)+1$. If $p\notin D$, then $D_Y=\pi^*D$, so
\[
K_Y\cdot D_Y=(\pi^*K_X+E)\cdot\pi^*D=K_X\cdot D, \qquad D_Y^2=D^2.
\]
Therefore $\bar c_2(Y,D_Y)$ increases by one. If $p\in D$, then $D_Y=\pi^*D-E$. Using $\pi^*L\cdot E=0$ for every divisor $L$ on $X$ and $E^2=-1$, we calculate
\begin{align*}
K_Y\cdot D_Y&=(\pi^*K_X+E)\cdot(\pi^*D-E)=K_X\cdot D+1,\\
D_Y^2&=(\pi^*D-E)^2=D^2-1.
\end{align*}
The two corrections cancel, while $c_2$ again increases by one.
\end{proof}

\begin{corollary}\label{blowup-chi}
If $p\notin D$, we have 
\[
\chi\bigl(\mathcal{E}\!nd_0(\Tt_Y(-\log D_Y))\bigr)=\chi\bigl(\mathcal{E}\!nd_0(\Tt_X(-\log D))\bigr)-5.
\]
Under a blow-up at a point of $D$, the change is $-4$.
\end{corollary}

\begin{proof}
The holomorphic Euler characteristic of the structure sheaf is birationally invariant, so $\chi(\Oo_Y)=\chi(\Oo_X)$. In the off-boundary case, Proposition~\ref{blowup-chern} changes the two nonconstant terms in \eqref{RR-log} by $-1$ and $-4$, respectively; their total contribution is $-1-4=-5$. In the on-boundary case, $\bar c_1^2$ is unchanged and $-4\bar c_2$ changes by $-4$. These are the asserted formulas.
\end{proof}

\begin{corollary}\label{many-blowups}
Let us consider a sequence of blow-ups:
\[
(X_n,D_n)\xrightarrow{~\pi_n~}\cdots \xrightarrow{~\pi_2~}(X_1,D_1) \xrightarrow{~\pi_1~}(X_0,D_0)=(X,D)
\]
with $p_i\in X_{i-1}$ the center of $\pi_i$ and $D_i$ the strict transform of $D_{i-1}$ for each $i$. Suppose that exactly $r$ of the centers lie outside the corresponding boundary and exactly $s$ lie on it; that is,
\[
\#\{i \mid p_i\notin D_{i-1}\}=r, \qquad \#\{i \mid p_i\in D_{i-1}\}=s, \qquad n=r+s.
\]
If the following inequality holds
\[
5r+4s>\chi\bigl(X,\mathcal{E}\!nd_0(\Tt_X(-\log D))\bigr),
\]
then $\Tt_{X_n}(-\log D_n)$ admits a positive-dimensional deformation with fixed determinant. In particular, this conclusion holds after any sequence of blow-ups for which the displayed inequality is satisfied.
\end{corollary}

\begin{proof}
Iterating Corollary~\ref{blowup-chi} gives
\[
\chi\bigl(\mathcal{E}\!nd_0(\Tt_{X_n}(-\log D_n))\bigr)=\chi\bigl(\mathcal{E}\!nd_0(\Tt_X(-\log D))\bigr)-5r-4s.
\]
The displayed numerical hypothesis makes the right-hand side strictly negative. Condition~\textup{(1)} of Corollary~\ref{numerical}, followed by Proposition~\ref{kuranishi}, then yields a positive-dimensional fixed-determinant deformation.
\end{proof}

The preceding numerical statement can be strengthened by adapting the generalized-logarithmic-sheaf argument in \cite{HMPV}. We record the resulting one-point comparison. Since a bundle and its dual have the same endomorphism bundle, we work on the logarithmic cotangent side.

\begin{theorem}\label{log-blowup-comparison}
Let $(X,D)$ be a smooth projective surface pair, let $\pi\colon Y\rightarrow X$ be the blow-up of a point $p$, and let $D_Y$ be the strict transform of $D$. There is an exact sequence
\begin{equation}\label{end-blowup-push}
0\to \pi_*\mathcal{E}\!nd_0(\Omega_Y^1(\log D_Y)) \to \mathcal{E}\!nd_0(\Omega_X^1(\log D))\to Q_p\to 0,
\end{equation}
where $Q_p$ is supported at $p$. Moreover:
\begin{itemize}
\item [(i)] if $p\in D$, then $\mathrm{length}(Q_p)=3$ and $R^1\pi_*\mathcal{E}\!nd_0(\Omega_Y^1(\log D_Y))\cong\Oo_p$;
\item [(ii)] if $p\notin D$, then $\mathrm{length}(Q_p)=3$ and $\mathrm{length}(R^1\pi_*\mathcal{E}\!nd_0(\Omega_Y^1(\log D_Y)))=2$.
\end{itemize}
In particular, the natural maps give
\[
\mathrm{H}^1(X,\pi_*\mathcal{E}\!nd_0(\Omega_Y^1(\log D_Y)))\twoheadrightarrow \mathrm{H}^1(X,\mathcal{E}\!nd_0(\Omega_X^1(\log D)))
\]
and
\[
\mathrm{H}^1(X,\pi_*\mathcal{E}\!nd_0(\Omega_Y^1(\log D_Y)))\hookrightarrow \mathrm{H}^1(Y,\mathcal{E}\!nd_0(\Omega_Y^1(\log D_Y))).
\]
Consequently, we obtain
\begin{equation}\label{blowup-monotonicity}
\mathrm{h}^1(Y,\mathcal{E}\!nd_0(\Omega_Y^1(\log D_Y)))\geq \mathrm{h}^1(X,\mathcal{E}\!nd_0(\Omega_X^1(\log D))).
\end{equation}
If $\mathrm{H}^2(X,\mathcal{E}\!nd_0(\Omega_X^1(\log D)))=0$, then the inequality is strict: the increase is at least one when $p\in D$ and at least two when $p\notin D$.
\end{theorem}

\begin{proof}
We give the local calculation, using the elementary-transform description of~\cite[Propositions~2.7 and~2.9]{HMPV}. When $p\in D$, those results show that $\Omega_X^1(\log(D,p)):=\pi_*\Omega_Y^1(\log D_Y)$ is the elementary transformation
\[
0\to \Omega_X^1(\log(D,p)) \to  \Omega_X^1(\log D) \to \Oo_p \to 0,
\]
and their formal-functions argument gives $R^1\pi_*\Omega_Y^1(\log D_Y)=0$. They also prove
\[
\Omega_Y^1(\log D_Y)_{|E_\pi}\cong \Oo_{\PP^1}(-1) \oplus\Oo_{\PP^1}(1),
\]
where $E_\pi$ denotes the exceptional curve.

Choose local coordinates $(x,y)$ at $p$ with $D=(x=0)$ and use the frame $\left(\frac{1}{x}d\!x,d\!y\right)$ of $\Omega_X^1(\log D)$. A trace-free endomorphism $A=\begin{pmatrix}a&b\\ c&-a\end{pmatrix}$ extends to an endomorphism of $\Omega_Y^1(\log D_Y)$ on the blow-up if and only if
\[
b(p)=0,\qquad \frac{\partial b}{\partial x}(p)=0,\qquad \frac{\partial b}{\partial y}(p)=2a(p).
\]
To obtain these conditions, use the two blow-up charts
\[
(x,y)=(u,uv),\qquad (x,y)=(st,s).
\]
In the first chart substitute
\[
\frac{dx}{x}=\frac{du}{u},\qquad dy=v\,du+u\,dv,
\]
and in the second chart substitute
\[
\frac{dx}{x}=\frac{dt}{t}+\frac{ds}{s},\qquad dy=ds.
\]
Conjugating the matrix $A$ by these two change-of-frame matrices, putting the entries over a common denominator, and collecting the coefficients of $u^{-1}$ and $s^{-1}$ gives precisely the three displayed conditions. Their classes are respectively the constant term of $b$, its $x$-linear term, and $b_y-2a$ modulo $(x,y)$, so they are independent. This gives \eqref{end-blowup-push} with $\mathrm{length}(Q_p)=3$ in the on-boundary case.

Furthermore,
\[
\mathcal{E}\!nd_0(\Omega_Y^1(\log D_Y))_{|{E_\pi}}\cong \Oo_{\PP^1}\oplus\Oo_{\PP^1}(2)\oplus\Oo_{\PP^1}(-2).
\]
The theorem on formal functions shows from this splitting that $R^1\pi_*\mathcal{E}\!nd_0(\Omega_Y^1(\log D_Y))$ is a finite-length sheaf supported at $p$. Its length can be obtained without leaving an uncomputed inverse limit. Taking Euler characteristics in \eqref{end-blowup-push} gives
\[
\chi(X,\pi_*\mathcal{E}\!nd_0(\Omega_Y^1(\log D_Y)))=\chi(X,\mathcal{E}\!nd_0(\Omega_X^1(\log D)))-3.
\]
Leray gives
\[
\chi(Y,\mathcal{E}\!nd_0(\Omega_Y^1(\log D_Y)))=\chi(X,\pi_*\mathcal{E}\!nd_0(\Omega_Y^1(\log D_Y)))-\operatorname{length}(R^1\pi_*\mathcal{E}\!nd_0(\Omega_Y^1(\log D_Y))).
\]
Corollary~\ref{blowup-chi} says that the left side is $\chi(X,\mathcal{E}\!nd_0(\Omega_X^1(\log D)))-4$. Hence the length is one. A length-one coherent sheaf supported at the reduced point $p$ is $\Oo_p$, proving the asserted isomorphism.

If $p\notin D$, the calculation is the ordinary point-blow-up calculation near $p$. One has
\[
\Omega_Y^1(\log D_Y)_{|{E_\pi}}\cong \Oo_{\PP^1}(-2)\oplus\Oo_{\PP^1}(1).
\]
Writing the blow-up in the same two affine charts and using the ordinary frames $(d\!x,d\!y)$ gives three independent first-jet conditions on a trace-free $2\times2$ matrix. Thus the quotient $Q_p$ again has length three. Formal functions applied to
\[
\mathcal{E}\!nd_0(\Omega_Y^1(\log D_Y))_{|{E_\pi}} \cong \Oo_{\PP^1}\oplus\Oo_{\PP^1}(3)\oplus\Oo_{\PP^1}(-3)
\]
shows that the higher direct image is supported at $p$. Its length is again forced by Euler characteristics: now Corollary~\ref{blowup-chi} gives
\[
\chi(Y,\mathcal{E}\!nd_0(\Omega_Y^1(\log D_Y)))=\chi(X,\mathcal{E}\!nd_0(\Omega_X^1(\log D)))-5,
\]
whereas the pushforward sequence contributes $-3$. Leray therefore contributes the remaining $-2$, so
$\mathrm{length}(R^1\pi_*\mathcal{E}\!nd_0(\Omega_Y^1(\log D_Y)))=2$. This is the rank-two logarithmic counterpart of Peternell's point-blow-up calculation.

Taking cohomology in \eqref{end-blowup-push}, and using that $Q_p$ has zero higher cohomology, gives the first displayed surjection. The five-term Leray sequence gives the displayed injection. This proves \eqref{blowup-monotonicity}. If $\mathrm{H}^2(X,\mathcal{E}\!nd_0(\Omega_X^1(\log D)))=0$, then also $\mathrm{H}^2(X,\pi_*\mathcal{E}\!nd_0(\Omega_Y^1(\log D_Y)))=0$, so the full stalk of $R^1\pi_*\mathcal{E}\!nd_0(\Omega_Y^1(\log D_Y))$ contributes to $\mathrm{H}^1(Y,\mathcal{E}\!nd_0(\Omega_Y^1(\log D_Y)))$.
\end{proof}

\begin{corollary}\label{compatible-blowups}
Suppose that $(Y,D_Y)\to(X,D)$ is obtained by a nonempty sequence of point blow-ups such that the reduced strict transform of the boundary remains smooth at every stage.
\begin{enumerate}[label=\textup{(\arabic*)}]
\item If $\Tt_X(-\log D)$ is infinitesimally nonrigid, then $\Tt_Y(-\log D_Y)$ is infinitesimally nonrigid.
\item If $\mathrm{H}^2\bigl(X,\mathcal E\!nd_0(\Tt_X(-\log D)) \bigr)=0$, then
\[
\mathrm{h}^1\bigl( Y,\mathcal E\!nd_0(\Tt_Y(-\log D_Y)) \bigr)>\mathrm{h}^1\bigl(X,\mathcal E\!nd_0(\Tt_X(-\log D))\bigr).
\]
\end{enumerate}
\end{corollary}

\begin{proof}
Write the sequence as $(X_i,D_i)$ for $0\leq i\leq n$, with $(X_0,D_0)=(X,D)$ and $(X_n,D_n)=(Y,D_Y)$, as in Corollary \ref{many-blowups}. At the $i$-th blow-up,
Theorem~\ref{log-blowup-comparison} gives
\[
\mathrm{h}^1\bigl(X_{i+1},\mathcal{E}\!nd_0(\Tt_{X_{i+1}}(-\log D_{i+1}))\bigr) \geq \mathrm{h}^1\bigl(X_i,\mathcal{E}\!nd_0(\Tt_{X_i}(-\log D_i))\bigr).
\]
Induction shows that a positive initial group remains positive on $Y$. If the initial group is zero but the displayed $\mathrm{H}^2$-group also vanishes, Theorem~\ref{log-blowup-comparison} also says that the first blow-up has first cohomology of dimension at least one for a center on the boundary and at least two for a center away from it. Applying the weak inequality to all remaining blow-ups preserves this nonvanishing.
\end{proof}

The Key Restriction Lemma of~\cite[Lemma~2.2]{HMPV} gives a uniform version of the exceptional-curve calculation, including the cases in which contracting the exceptional curve creates a singular boundary.

\begin{proposition}\label{key-lemma-exceptional}
Let $(Y,D)$ be a smooth surface pair, and let $C\cong \PP^1$ be a $(-1)$-curve not contained in $D$. Write the scheme-theoretic intersection as
\[
D\cap C=\sum_{i=1}^{n}a_i p_i, \qquad m=\sum_{i=1}^{n}a_i,
\]
where the $p_i$ are distinct and $a_i\geq1$. Then there is an exact sequence
\begin{equation}\label{key-exceptional-sequence}
0\to \Oo_{\PP^1}(2-n) \to \Tt_Y(-\log D)_{|C} \to \Oo_{\PP^1}(n-m-1) \to 0.
\end{equation}
If $m+4\geq2n$ (in particular, if $n\leq4$), the sequence splits.
\end{proposition}

\begin{proof}
Set $B=(D\cap C)_{\mathrm{red}}=p_1+\cdots+p_n$, so that $\deg B=n$. The Key Restriction Lemma gives
\[
0\to \Tt_C(-\log B)\cong \Oo_{\PP^1}(2-n) \to \Tt_Y(-\log D)_{|C} \to \Nn^{\log}_{C|Y} \to 0.
\]
The same lemma compares the logarithmic normal sheaf $\Nn^{\log}_{C|Y}$ with $\Nn_{C|Y}\cong \Oo_{\PP^1}(-1)$ and shows that the quotient has length $\sum_i(a_i-1)=m-n$. Consequently, we obtain
\[
\Nn^{\log}_{C|Y}\cong \Oo_{\PP^1}(n-m-1),
\]
which proves \eqref{key-exceptional-sequence}. Its extension group is $\mathrm{H}^1\bigl(\PP^1,\Oo_{\PP^1}(m+3-2n)\bigr)$, and this vanishes when $m+3-2n\geq-1$.
\end{proof}

For transverse intersections, one has $m=n$, and Proposition \ref{key-lemma-exceptional} gives, for $n\leq4$,
\begin{equation}\label{transverse-exceptional-splitting}
\Tt_Y(-\log D)_{|C} \cong \Oo_{\PP^1}(2-n)\oplus\Oo_{\PP^1}(-1).
\end{equation}
Thus the familiar cases $n=0$ and $n=1$ are $\Oo_{\PP^1}(2)\oplus\Oo_{\PP^1}(-1)$ and $\Oo_{\PP^1}(1)\oplus\Oo_{\PP^1}(-1)$, respectively. The remaining cases are
\[
\begin{array}{c|c|c|c}
 n & 2 & 3 & 4 \\ \hline
 \Tt_Y(-\log D)_{|C}
 & \Oo_{\PP^1}\oplus\Oo_{\PP^1}(-1)
 &\Oo_{\PP^1}(-1)\oplus\Oo_{\PP^1}(-1)
 & \Oo_{\PP^1}(-2)\oplus\Oo_{\PP^1}(-1)
\end{array}
\]
Unlike $n=0,1$, these three splitting types have 
\[
\mathrm{H}^1\bigl(C,\mathcal{E}\!nd_0(\Tt_Y(-\log D)_{|C})\bigr)=0.
\]
When $C$ is contracted, its $n$ transverse intersection points with $D$ map to a single point of the image boundary: this is an ordinary node for $n=2$ and an ordinary $n$-fold point for $n\geq3$. The preceding vanishing shows that the restriction to $C$ supplies no nonzero first-cohomology class supported on the exceptional curve. Thus, in these cases, nonrigidity cannot be deduced from the exceptional curve alone; one must instead study the deformations of the generalized logarithmic sheaf associated with the singular image boundary.

If $n=1$ and the contact order is $m\geq1$, the sequence always splits as
\[
\Tt_Y(-\log D)_{|C}\cong \Oo_{\PP^1}(1)\oplus\Oo_{\PP^1}(-m).
\]
Hence we obtain $\mathrm{h}^1\bigl(C,\mathcal{E}\!nd_0(\Tt_Y(-\log D)_{|C})\bigr)=m$. Thus a tangential one-branch contraction has nonzero first cohomology on the exceptional curve, whereas a transverse contraction with two, three, or four branches has none. This is a distinction at the level of the exceptional restriction. To turn the former classes into nonzero local higher-direct-image classes, one must additionally compare the cohomology of $C$ with that of its infinitesimal neighborhoods. This comparison is the local issue that enters the classification on nonminimal surfaces.


\subsection{The ordinary nodal contraction}

We determine the local comparison in the first transverse multibranch case. Let $\sigma\colon Y\rightarrow X$ contract a $(-1)$-curve $C$, and suppose that a smooth divisor $D\subset Y$ meets $C$ transversely at two distinct points. Then $\Delta=\sigma(D)$ has an ordinary node at $p=\sigma(C)$. We use the locally free logarithmic sheaf at the node; analytically, the local model is $(xy=0)\subset\AA^2$. As an analogue of Theorem~\ref{log-blowup-comparison}, we obtain the following.

\begin{theorem}\label{nodal-contraction}
There are two exact sequences
\begin{equation}\label{nodal-cotangent-transform}
0\to \sigma_{\!*}\Omega_Y^1(\log D)\to \Omega_X^1(\log \Delta) \to \Oo_p\to 0
\end{equation}
and
\begin{equation}\label{nodal-end-transform}
0\to \sigma_{\!*}\mathcal{E}\!nd_0(\Omega_Y^1(\log D))\to \mathcal{E}\!nd_0(\Omega_X^1(\log \Delta))\to \Oo_p\to 0
\end{equation}
with $R^1\!\sigma_{\!*}\mathcal{E}\!nd_0(\Omega_Y^1(\log D))=0$. Consequently, we obtain
\begin{equation}\label{nodal-h1-formula}
\mathrm{h}^1(Y,\mathcal{E}\!nd_0(\Omega_Y^1(\log D)))=\mathrm{h}^1(X,\mathcal{E}\!nd_0(\Omega_X^1(\log \Delta)))+1-\rho_p,
\end{equation}
where $\rho_p\in\{0,1\}$ is the rank of the evaluation map
\[
\mathrm{H}^0(X,\mathcal{E}\!nd_0(\Omega_X^1(\log \Delta)))\to \CC_p
\]
induced by \eqref{nodal-end-transform}. In particular, genuine nonrigidity of $(X,\Delta)$ implies genuine nonrigidity of $(Y,D)$. If $\Omega_X^1(\log \Delta)$ is simple so that $\rho_p=0$, then the nodal resolution creates exactly one new first-order deformation.
\end{theorem}

\begin{proof}
Work analytically at $p$ with coordinates $(x,y)$ such that $\Delta=(xy=0)$. The logarithmic cotangent bundle $\Omega_X^1(\log \Delta)$ has frame $\langle \alpha=\frac{dx}{x},\beta=\frac{dy}{y}\rangle$. On the first blow-up chart write $x=u$, $y=uv$. The strict transform of the boundary in this chart is $(v=0)$, and $\Omega_Y^1(\log D)$ has frame $\langle du,\frac{dv}{v}\rangle$. Since
\[
\alpha=\frac{du}{u} \quad \text{and}\quad \beta=\frac{du}{u}+\frac{dv}{v},
\]
a section $a\alpha+b\beta$ extends across the exceptional curve if and only if $a+b$ vanishes at the origin. The second chart, $x=st$, $y=t$, gives the same condition. Thus we obtain
\[
\sigma_*\Omega_Y^1(\log D)=\{a\alpha+b\beta\mid (a+b)(p)=0\},
\]
which proves \eqref{nodal-cotangent-transform}.

For the endomorphism calculation, write an endomorphism of $\Omega_X^1(\log \Delta)$ in the frame $(\alpha,\beta)$ as $A=\begin{pmatrix}a&b\\c&d\end{pmatrix}$. On the chart $(u,v)$, the change-of-frame matrix from $(\alpha,\beta)$ to $(du,dv/v)$ is $P=\begin{pmatrix}u^{-1}&u^{-1}\\0&1\end{pmatrix}$. The only possibly singular entry of $PAP^{-1}$ is $(-a+b-c+d)/u$. The second blow-up chart gives the identical condition. Hence $A$ extends to an endomorphism of $\Omega_Y^1(\log D)$ precisely when $(-a+b-c+d)(p)=0$. After imposing $d=-a$, this remains one nonzero linear condition on trace-free endomorphisms. This proves \eqref{nodal-end-transform}.

Proposition~\ref{key-lemma-exceptional}, with $n=m=2$, gives $\Tt_Y(-\log D)_{|C}\cong \Oo_C\oplus\Oo_C(-1)$. Therefore
\[
\mathcal{E}\!nd_0(\Omega_Y^1(\log D))_{|C}\cong\Oo_C\oplus\Oo_C(1)\oplus\Oo_C(-1).
\]
All these summands, and all their twists by the successive conormal quotients $\Oo_C(k)$ for $k\geq0$, have vanishing first cohomology. The theorem on formal functions therefore gives vanishing $R^1\sigma_*\mathcal{E}\!nd_0(\Omega_Y^1(\log D))=0$.

Leray now identifies
\[
\mathrm{H}^1(Y,\mathcal{E}\!nd_0(\Omega_Y^1(\log D)))\simeq \mathrm{H}^1(X,\sigma_*\mathcal{E}\!nd_0(\Omega_Y^1(\log D))).
\]
Taking cohomology in \eqref{nodal-end-transform} proves \eqref{nodal-h1-formula}. If $\Omega_X^1(\log \Delta)$ is simple, then $\mathrm{H}^0(X,\mathcal{E}\!nd_0(\Omega_X^1(\log \Delta)))=0$, so $\rho_p=0$.
\end{proof}

\begin{corollary}\label{nodal-rigidity}
In the nodal contraction setting, the resolved smooth pair $(Y,D)$ is genuinely rigid if and only if the nodal logarithmic bundle $\Omega_X^1(\log \Delta)$ is genuinely rigid and the evaluation map
\[
\mathrm{H}^0(X,\mathcal{E}\!nd_0(\Omega_X^1(\log \Delta)))\to \CC_p
\]
is surjective. In particular, a simple nodal logarithmic bundle never gives a rigid resolved pair.
\end{corollary}

\begin{proof}
Formula \eqref{nodal-h1-formula} reads
\[
\mathrm{h}^1(Y,\mathcal{E}\!nd_0(\Omega_Y^1(\log D)))=\mathrm{h}^1(X,\mathcal{E}\!nd_0(\Omega_X^1(\log \Delta)))+1-\rho_p,
\]
where $\rho_p$ is the rank of a linear map to the one-dimensional vector space $\CC_p$. Hence $\rho_p$ is either zero or one. The left side vanishes exactly when the first term on the right vanishes and $\rho_p=1$, which is equivalent to surjectivity of the evaluation map. If $\Omega_X^1(\log \Delta)$ is simple, then $H^0(X,\End_0(\Omega_X^1(\log \Delta)))=0$, so the evaluation map has rank zero and the formula gives 
\[
\mathrm{h}^1(Y,\mathcal{E}\!nd_0(\Omega_Y^1(\log D)))=\mathrm{h}^1(X,\mathcal{E}\!nd_0(\Omega_X^1(\log \Delta)))+1>0.\qedhere
\]
\end{proof}


\subsection{The ordinary triple-point contraction}

We next consider $n=m=3$. In contrast with the nodal case, the endomorphism bundle is unchanged by resolving an ordinary triple point.

\begin{theorem}\label{triple-contraction}
Let $\sigma\colon Y\rightarrow X$ contract a $(-1)$-curve $C\cong \PP^1$. Suppose that a smooth divisor $D\subset Y$ meets $C$ transversely at three distinct points and that $\Delta=\sigma(D)$ has an ordinary triple point at $p=\sigma(C)$. Then we obtain
\begin{equation}\label{triple-bundle-transform}
\Omega_Y^1(\log D)\cong \sigma^*\Omega_X^1(\log \Delta)\otimes\Oo_Y(-C).
\end{equation}
Consequently, we have
\begin{equation}\label{triple}
\begin{cases}
~~\mathcal{E}\!nd_0(\Omega_Y^1(\log D))\cong \sigma^*\mathcal{E}\!nd_0(\Omega_X^1(\log \Delta));\\
~~\sigma_{\!*}\mathcal{E}\!nd_0(\Omega_Y^1(\log D))\cong \mathcal{E}\!nd_0(\Omega_X^1(\log \Delta));\\
~~R^1\!\sigma_{\!*}\mathcal{E}\!nd_0(\Omega_Y^1(\log D))=0.
\end{cases}
\end{equation}
In particular, we obtain
\begin{equation}\label{triple-h1-invariance}
\mathrm{H}^1(Y,\mathcal{E}\!nd_0(\Omega_Y^1(\log D)))\cong  \mathrm{H}^1(X,\mathcal{E}\!nd_0(\Omega_X^1(\log \Delta))),
\end{equation}
and thus resolving an ordinary triple point preserves genuine infinitesimal rigidity and nonrigidity.
\end{theorem}

\begin{proof}
The assertion is local at $p$. Choose analytic coordinates so that the three tangent directions are represented by $f=xy(x-y)=0$. The module of logarithmic derivations is free, with basis
\[
\delta_1=x\frac{\partial}{\partial x}+y\frac{\partial}{\partial y},\qquad \delta_2=f_y\frac{\partial}{\partial x}-f_x\frac{\partial}{\partial y}.
\]
Indeed, $\delta_1(f)=3f$, $\delta_2(f)=0$, and the determinant of the coefficient matrix is $-3f$.

On the blow-up chart $x=u$, $y=uv$, the strict transform is $v(1-v)=0$. The lifted logarithmic derivations are
\[
\widetilde\delta_1=u\frac{\partial}{\partial u}, \qquad 
\widetilde\delta_2=u\left( u(1-2v)\frac{\partial}{\partial u}-3v(1-v)\frac{\partial}{\partial v} \right).
\]
Both contain exactly one factor of the local equation $u$ of the exceptional curve. After removing this common factor, the two vector fields form a frame of $\Tt_Y(-\log D)$: their determinant is a unit times $v(1-v)$, the reduced equation of the strict transform. The same calculation on the second blow-up chart gives
\[
\Tt_Y(-\log D)\cong \sigma^*\Tt_X(-\log\Delta)\otimes\Oo_Y(C).
\]
Dualizing proves \eqref{triple-bundle-transform}. The line-bundle twist cancels in the endomorphism bundle, giving the first of \eqref{triple}. Since
\[
\sigma_{\!*}\Oo_Y=\Oo_X, \qquad R^1\sigma_{\!*}\Oo_Y=0,
\]
the projection formula proves the rest of \eqref{triple}, and Leray then gives \eqref{triple-h1-invariance}. Note that the Key Restriction Lemma gives
\[
\Tt_Y(-\log D)_{|C}\cong \Oo_C(-1)^{\oplus2},
\]
which is also the restriction of $\sigma^*\Tt_X(-\log\Delta)\otimes \Oo_Y(C)$ to $C$.
\end{proof}

\subsection{The ordinary quadruple-point contraction}

The case $n=m=4$ is again an elementary modification, although its local origin differs from the nodal case.

\begin{theorem}\label{quadruple-contraction}
Let $\sigma\colon Y\rightarrow X$ contract a $(-1)$-curve $C\cong \PP^1$. Suppose that a smooth divisor $D\subset Y$ meets $C$ transversely at four distinct points and that $\Delta=\sigma(D)$ has an ordinary quadruple point at $p=\sigma(C)$. Then we have an exact sequence
\begin{equation}\label{quadruple-end-transform}
0\to \sigma_*\mathcal{E}\!nd_0(\Tt_Y(-\log D)) \to \mathcal{E}\!nd_0(\Tt_X(-\log \Delta))\to \Oo_p\to 0
\end{equation}
and $R^1\sigma_*\mathcal{E}\!nd_0(\Tt_Y(-\log D))=0$. Consequently,
\begin{equation}\label{quadruple-h1-formula}
\mathrm{h}^1(Y,\mathcal{E}\!nd_0(\Tt_Y(-\log D))) =\mathrm{h}^1(X,\mathcal{E}\!nd_0(\Tt_X(-\log \Delta)))+1-\rho_p^{(4)},
\end{equation}
where $\rho_p^{(4)}\in\{0,1\}$ is the rank of the evaluation map induced by \eqref{quadruple-end-transform}. In particular, genuine nonrigidity of $\Tt_X(-\log \Delta)$ implies genuine nonrigidity of $\Tt_Y(-\log D)$, and a simple logarithmic tangent bundle $\Tt_X(-\log \Delta)$ acquires exactly one new deformation after resolution.
\end{theorem}

\begin{proof}
Locally, $\Delta$ is the cone over four distinct points of $\PP^1$. Choose a reduced homogeneous equation $f(x,y)=\prod_{i=1}^4\ell_i(x,y)$. As in the triple-point case, a basis of logarithmic derivations is
\[
\delta_1=x\partial_x+y\partial_y, \qquad \delta_2=f_y\partial_x-f_x\partial_y.
\]
On the blow-up chart $x=u$, $y=uv$, the lift of $\delta_1$ contains exactly one factor $u$, whereas the lift of $\delta_2$ contains exactly two factors $u$. If $e_1,e_2$ denote a local frame of $\Tt_Y(-\log D)$, then the pullbacks of the generators have the form
\[
\sigma^*\delta_1=u e_1, \qquad \sigma^*\delta_2=u^2 e_2,
\]
up to multiplication and triangular change of basis by units. Let $A=\begin{pmatrix}a&b\\c&-a\end{pmatrix}$ be a trace-free endomorphism of $\Tt_X(-\log \Delta)$ in the displayed frame. Conjugating by $\mathrm{diag}(u,u^2)$ shows that its only possibly singular entry is $b/u$. The calculation on the second chart shows that extension over $C$ is equivalent to $b\in\mathfrak m_p$. This is one independent residue condition and proves \eqref{quadruple-end-transform}.

The Key Restriction Lemma gives $\Tt_Y(-\log D)_{|C}\cong \Oo_C(-2)\oplus\Oo_C(-1)$, and so we obtain
\[
\mathcal{E}\!nd_0(\Tt_Y(-\log D))_{|C} \cong \Oo_C\oplus\Oo_C(1)\oplus\Oo_C(-1).
\]
After tensoring with every successive conormal quotient $\Oo_C(k)$, $k\geq0$, the first cohomology still vanishes. Formal
functions therefore gives $R^1\sigma_{\!*}\mathcal{E}\!nd_0(\Tt_Y(-\log D))=0$. Cohomology of \eqref{quadruple-end-transform}, followed by Leray, proves \eqref{quadruple-h1-formula}.
\end{proof}

\subsection{A unibranch tangential contraction}

We next examine a smooth branch having contact order $m\geq2$ with the exceptional curve. This is the first case in which the Key Restriction Lemma does not by itself determine the higher direct image of the endomorphism bundle.

\begin{proposition}\label{cusp-local-calculation}
Let $\sigma\colon Y\rightarrow X$ contract a $(-1)$-curve $C\cong \PP^1$, and suppose that a smooth divisor $D\subset Y$ meets $C$ at one point with intersection multiplicity $m\geq3$. Assume, in addition, that the image branch $\Delta=\sigma(D)$ is the quasihomogeneous cusp with local equation
\begin{equation}\label{cusp-model}
y^m-x^{m+1}=0.
\end{equation}
Then we have $\Tt_Y(-\log D)_{|C} \cong \Oo_C(1)\oplus \Oo_C(-m)$, and so 
\begin{equation}\label{cusp-exceptional-h1}
\mathrm{h}^1\bigl(C,\mathcal{E}\!nd_0(\Tt_Y(-\log D)_{|C})\bigr)=m.
\end{equation}
The pullbacks of a basis downstairs have different exceptional orders: in the blow-up chart $x=u$, $y=uv$, they are
\begin{equation}\label{cusp-lift}
\widetilde\delta_1 =mu\partial_u+v\partial_v, \qquad  \widetilde\delta_2=u^{m-2}\left( mu v^{m-1}\partial_u+\bigl((m+1)u-mv^m\bigr)\partial_v \right).
\end{equation}
\end{proposition}

\begin{proof}
The model \eqref{cusp-model} is obtained by blowing down the smooth branch whose equation on the blow-up is $u=v^m$. For $f=y^m-x^{m+1}$, logarithmic derivations downstairs have the basis
\[
\delta_1=mx\partial_x+(m+1)y\partial_y, \qquad \delta_2=f_y\partial_x-f_x\partial_y=my^{m-1}\partial_x+(m+1)x^m\partial_y.
\]
Substituting $x=u$, $y=uv$ gives \eqref{cusp-lift}. The Key Restriction Lemma, with $n=1$ and total contact $m$, gives the splitting for $\Tt_Y(-\log D)_{|C}$, which induces
\[
\mathcal{E}\!nd_0(\Tt_Y(-\log D)_{|C})\cong \Oo_C\oplus\Oo_C(m+1)\oplus\Oo_C(-m-1).
\]
So \eqref{cusp-exceptional-h1} follows from $\mathrm{h}^1(\PP^1,\Oo_{\PP^1}(-m-1))=m$.
\end{proof}

The formal-functions step now differs essentially from the transverse cases. If $C^{(k)}$ denotes the $k$-th infinitesimal neighborhood of $C$, the successive conormal quotients introduce the groups $\mathrm{H}^1\bigl(C,\Oo_C(k-m-1)\bigr)$ for $0\leq k\leq m-1$, which do not vanish. Consequently, restriction to $C$ does not determine either the length or the module structure of
\[
R^1\sigma_{\!*}\mathcal{E}\!nd_0(\Tt_Y(-\log D)).
\]
One must compute the transition maps
\[
\mathrm{H}^1\bigl(C^{(k+1)},\mathcal{E}\!nd_0(\Tt_Y(-\log D)_{|{C^{(k+1)}}})\bigr) \to \mathrm{H}^1\bigl(C^{(k)},\mathcal{E}\!nd_0(\Tt_Y(-\log D)_{|{C^{(k)}}})\bigr)
\]
using the explicit lifts \eqref{cusp-lift}. This is the first unresolved local calculation in the contraction program; assigning a length from the splitting type alone would be unjustified.

\begin{problem}\label{cusp-formal-neighborhood}
For the local model $y^m=x^{m+1}$, determine
\[
\sigma_{\!*}\mathcal{E}\!nd_0(\Omega_Y^1(\log D)), \qquad R^1\sigma_{\!*}\mathcal{E}\!nd_0(\Omega_Y^1(\log D)),
\]
including the complete finite-length module structure at the cusp.
\end{problem}

\begin{remark}
The preceding case-by-case calculations do not constitute a classification on nonminimal surfaces. They determine the local comparison for ordinary transverse contractions with two, three, or four branches. Even in those cases, a global rigidity statement still requires the corresponding generalized logarithmic sheaf on the contracted surface to be understood.  

The formal-neighborhood problem above remains open in the present argument for a tangential unibranch contraction. Also not treated are transverse contractions with at least five branches, contractions with both tangencies and several branches,
nonordinary singular image boundaries, and sequences in which successive exceptional configurations interact.  Accordingly, these results are local ingredients for, rather than the completion of, the classification on nonminimal surfaces.
\end{remark}


\end{document}